\documentclass[12pt,reqno]{amsart}
\usepackage{amsmath,amsfonts,amsthm,amssymb,color}
\usepackage{cmmib57}
\usepackage{exscale}
\usepackage{amscd}
\usepackage{latexsym}
\usepackage{hyperref}
\usepackage{graphicx}
\usepackage[T1]{fontenc}
\usepackage[latin1]{inputenc}
\usepackage{pdfsync}

\usepackage[width=6.5in, left=1in, right=1in, height=8.8in, top=1.1in,bottom=1.1in]{geometry}

\definecolor{rp}{rgb}{0.25, 0, 0.75}
\definecolor{dg}{rgb}{0, 0.5, 0}

\newcommand{\der}{\delta}
\newcommand{\hf}{\hat F}
\newcommand{\hu}{\hat u}
\newcommand{\hpsi}{\hat\psi}

\newcommand{\hal}{\hat{\alpha}}
\newcommand{\hro}{\hat{\rho}}
\newcommand{\tf}{\widetilde F}
\newcommand{\tu}{\widetilde u}

\newcommand{\id}{\mbox{Id}}

\newcommand{\ist}{\int_{s}^{t}}

\newcommand{\norm}[1]{\lVert #1\rVert}

\newcommand{\ott}{[0,T]}

\newcommand{\2}{\mathbf{2}}
\newcommand{\xd}{{\bf x^{2}}}

\newcommand{\Id}{\mathfrak{t}}

\newcommand{\cb}{{\mathcal B}}
\newcommand{\cac}{{\mathcal C}}

\newcommand{\ci}{{\mathcal I}}
\newcommand{\cj}{{\mathcal J}}

\newcommand{\cn}{{\mathcal N}}

\newcommand{\cq}{{\mathcal Q}}
\newcommand{\crr}{{\mathcal R}}
\newcommand{\cs}{{\mathcal S}}
\newcommand{\ct}{{\mathcal T}}

\newcommand{\cz}{{\mathcal Z}}

\newcommand{\al}{\alpha}
\newcommand{\bet}{\beta}
\newcommand{\ep}{\varepsilon}

\newcommand{\ga}{\gamma}
\newcommand{\ka}{\kappa}
\newcommand{\la}{\lambda}
\newcommand{\laa}{\Lambda}

\newcommand{\si}{\sigma}
\newcommand{\te}{\theta}

\newcommand{\vp}{\varphi}

\newcommand{\N}{{\mathbb N}}

\newcommand{\R}{{\mathbb R}}
\newcommand{\X}{{\mathbb X}}

\newcommand{\bx}{\mathbf{x}}

\newcommand{\lla}{\left\langle}
\newcommand{\rra}{\right\rangle}
\newcommand{\lcl}{\left\{}
\newcommand{\rcl}{\right\}}
\newcommand{\lp}{\left(}
\newcommand{\rp}{\right)}
\newcommand{\lc}{\left[}
\newcommand{\rc}{\right]}
\newcommand{\lln}{\left|}
\newcommand{\rrn}{\right|}

\newtheorem{thm}{Theorem}[section]
\newtheorem{cor}[thm]{Corollary}
\newtheorem{lem}[thm]{Lemma}
\newtheorem{prop}[thm]{Proposition}
\newtheorem{notation}[thm]{Notation}
\newtheorem{defn}[thm]{Definition}
\newtheorem{hyp}[thm]{Hypothesis}

\theoremstyle{remark}
\newtheorem{rem}[thm]{Remark}

\begin{document}

\title[Controlled fully nonlinear SPDEs]{Controlled viscosity solutions\\ of fully nonlinear rough PDEs}

\author{M. Gubinelli \and S. Tindel \and I. Torrecilla}

\address{Massimiliano Gubinelli,
CEREMADE,
Universit\'e de Paris-Dauphine,
75116 Paris, France.}
\email{massimiliano.gubinelli@ceremade.dauphine.fr}

\address{Samy Tindel, Institut {\'E}lie Cartan Lorraine,
Universit\'e de Lorraine, B.P. 239,
54506 Vand{\oe}u\-vre-l{\`e}s-Nancy, France.}
\email{samy.tindel@univ-lorraine.fr}

\address{Iv\'an Torrecilla, Institut {\'E}lie Cartan Lorraine,
Universit\'e de Lorraine,  B.P. 239,
54506 Vand{\oe}u\-vre-l{\`e}s-Nancy, and Universit\'e de Paris-Dauphine,
75116 Paris, France.}
\email{itorrecillatarantino@gmail.com}

\thanks{S. Tindel is member of the BIGS
(Biology, Genetics and Statistics) team at INRIA. I. Torrecilla wishes to acknowledge the support from both BIGS
(Biology, Genetics and Statistics) team at INRIA and the project entitled "Explorations des Chemins
RUgeux (ECRU)/ Explorations on rough paths" ANR Projet Blanc
2009/2012 by means of a post-doctoral position for six months in each one.}


\begin{abstract}
We propose a definition of viscosity solutions to fully nonlinear PDEs driven by a rough path via  appropriate notions of test functions and rough jets. These objects will be defined as controlled processes with respect to the driving rough path. We show that this notion is compatible with the seminal results of Lions and Souganidis~\cite{LS98,LSo1} and with the recent results of Friz and coauthors~\cite{CFO} on fully non-linear SPDEs with rough drivers.
\end{abstract}



\date{\today}

\maketitle


\section{Introduction}
This note focuses on fully nonlinear rough partial differential equations with general form
\begin{align}
\label{main:spde1} \left\{
\begin{array}{ll}
d u =F\lp t,\te,u,Du,D^{2}u \rp dt
+\sum_{l=1}^{d} \si^{l}\lp t,\te, u, Du\rp dx_t^{l},&\hspace{-0.1cm}(t, \te)\in[0,T]\times \R^n    \\
u(0,\te)=\alpha(\te), & \te\in \R^n,
\end{array}
\right.
\end{align}
where $T>0$ is a fixed time horizon,  $F$ is a function defined on $[0,T]\times \R^n\times\R\times\R^n$ satisfying some suitable smoothness and monotonicity conditions, $\si$ is a smooth and bounded coefficient,  and $\alpha$ is a regular enough initial condition.

\smallskip

When the noisy driver $x$ is H\"older continuous, recent advances have allowed to solve equation \eqref{main:spde1} in several special cases of interest:
\begin{itemize}
\item
The case of coefficients $F(Du,D^{2}u )$ and $\si(Du)$ is handled in \cite{LS98}, when the driver $x$ is a $d$-dimensional Brownian motion.
\item
The article \cite{LSo1} focuses on a dependence of the form $F(Du,D^{2}u )$ and $\si(u)$, still in the case of a $d$-dimensional Brownian motion $x$.
\item
The situation where $x$ is a general H\"older continuous function generating a rough path, with a function $F$ depending on $t,\te,Du,D^{2}u$ and a family of functions $\si(x,Du)$ depending linearly on $Du$, is treated in \cite{CFO} and further analysed and applied in~\cite{DF,FO,CDFO,FO2}.
\end{itemize}
Let us also point out that the global strategy in all those papers  can be briefly outlined in the following way (though it might be somehow implicit in the series of papers \cite{LS98,LSo1}):

\smallskip

\noindent
\textit{(i)} Start from a regularization $x^{\ep}$ of $x$, for which equation \eqref{main:spde1} is interpreted in the deterministic viscosity sense. Let us call $u^{\ep}$ its solution.

\smallskip

\noindent
\textit{(ii)}
Perform some natural change of variables allowing to transform equation \eqref{main:spde1} driven by $x^{\ep}$ into a deterministic type equation whose coefficients depend on $x^{\ep}$, or possibly on a rough path above $x^{\ep}$. Call $v^{\ep}$ the solution to the transformed equation, which enjoys two crucial properties: $v^{\ep}=G_{1}(u^{\ep})$ on the one hand, and  $v^{\ep}=G_{2}(\X^{\ep})$ on the other hand, where $\X^{\ep}$ stands for a geometric rough path above $x^{\ep}$. We are obviously very loose about the definition of $G_{1},G_{2}$, but let us mention that $G_{1}$ should be invertible, and that $G_{2}$ should be continuous with respect to the rough path topology.

\smallskip

\noindent
\textit{(iii)}
With those elements in hand, we now assume that $\X^{\ep}$ converges to a rough path above $x$ as $\ep\to 0$, and the solution to equation \eqref{main:spde1} is defined as $u=G_{1}^{-1}(G_{2}(\X))$.

\smallskip

\noindent
By its very construction, this way of defining a solution is thus an extension of the viscosity solution for equation \eqref{main:spde1} with a smooth driver.

\smallskip

We propose to step back and wonder what is really meant by a solution to equation \eqref{main:spde1} when $x$ is a rough signal. More specifically,  think of~(\ref{main:spde1}) as an integral type equation of the form:
\begin{equation}
\label{main:spde2} u(t,\te)=\alpha(\te)+\int_0^t F\lp r,\te,(u,Du)(r,\te)\rp dr+\sum_{l=1}^d\int_0^t\si^l\lp r,\te,(u, Du)(r,\te)\rp dx_r^l,
\end{equation}
where the integral with respect to $x$ will be understood in the rough path sense. We wish to address the following questions:
\begin{itemize}
\item
Give a general notion of viscosity solution, based on a suitable class of test functions, which will be called $\ct_{\si}$.
\item
Show that this notion of viscosity solution is a generalization of the notion of strong solution.
\item
Also define some related jets, called $\mathfrak{P}_{\si}^{+}$ and  $\mathfrak{P}_{\si}^{-}$, and study their compatibility with the set of test functions.
\item
In some special cases, observe how equation \eqref{main:spde2} can be transformed into a PDE with noisy coefficients thanks to rough paths methods, and where our transforms also involve the set $\ct_{\si}$ and $\mathfrak{P}_{\si}^{\pm}$.
\end{itemize}
We shall thus observe how noisy fully nonlinear equations can be solved within a framework which mimics the deterministic setting, by looking at the effect of the functions $G_{1},G_{2}$ alluded to above on test functions and jets. This gives an alternative way to solve equations like \eqref{main:spde2}, but it should be noticed that our aim here is not to come up with new results in this direction (namely, our examples of application will be those of \cite{CFO,LSo1}). However, we believe that our article fills a gap in the theory, insofar as it describes a natural notion of solution for viscosity solutions to rough PDEs.

\smallskip

Let us give some hints about the way to define our class of test functions $\ct_{\si}$. A natural guess in view of equation \eqref{main:spde2} is to assume that those test functions $\psi\in\ct_{\si}$ satisfy the relation:
\begin{equation*}
\psi_{t}(\te) = \al(\te) + \int_{0}^{t} \si^l(r,\te,\psi_r,D\psi_r(\te)) \, dx_{r}^{l} +
\int_{0}^{t} \psi^{t}_{r}(\te) \, dr,
\end{equation*}
for a given continuous function $\psi^{t}$. However, this natural guess raises the problem of a proper definition of the rough integral appearing in the r.h.s.. A well-suited framework to settle this issue is that of \emph{(strongly) controlled paths} as introduced in~\cite{Gu}.  Controlled paths are functions whose increments ``looks like'' those of $x$ (or more generally that can be expanded along a basis of well behaved even if irregular objects). In our particular case the expansion for $\psi$ is specified in a way allowing to take also the drift $\psi^{t}$ into account. With this structure in mind, our main task will be the following: (i) Obtain some space-time Taylor type expansions for processes like $\psi$, which allows to relate the set of test functions $\ct_{\si}$ and the jets $\mathfrak{P}_{\si}^{\pm}$. (ii) Describe the composition of $\psi$ with a noisy flow driven by $x$, in order to define transformations of $\ct_{\si}$ into a space of test functions corresponding to a deterministic viscosity equation. These steps will be detailed in the remainder of the paper, but let us mention at this point that our analysis is restricted here to a noise $x$ which can be lifted to a rough path with H\"older regularity greater than $1/3$. This is made for sake of clarity, but rougher situations could be handled at the cost of higher order expansions.

Let us point out also that the approach we follow in defining suitable spaces of test functions or jets is quite natural. In the series of papers~\cite{BBM,BFM,BM} Buckdahn, Ma and coauthors develop a similar theory using stochastic integration and stochastic space-time Taylor series. Even comparing with these works the controlled approach has the great advantage of not requiring any condition of adaptedness or suitable stopping time to identify the extremal points needed in the viscosity formulation. This simplifies a lot the statement of the results and renders them quite similar to the deterministic theory while at the same time going well beyond the semi-martingale setting.

\smallskip

Here is a sketch of our paper: Section \ref{sec:algebraic-intg} is devoted to recall some basic notions of algebraic integration and rough flows. Section \ref{sec:viscosity-solutions} deals with general definitions for viscosity solutions to rough PDEs. We then focus on two particular cases: transport equations at Section~\ref{sec:transport} and semilinear stochastic dependence at Section~\ref{sec:semilinear}.

\smallskip

\noindent
\textbf{Notations:} throughout the paper we use the summation convention over repeated indices. For jets, the notation $\mathfrak{P}^{\pm}$ stands for the fact that either  $\mathfrak{P}^{+}$ or $\mathfrak{P}^{-}$ can be considered in the statement.

\section{Algebraic integration and rough paths equations}\label{sec:algebraic-intg}

\subsection{Increments}\label{incr}

In this section we present the basic  algebraic structures which
will allow us to define pathwise integrals with respect to
 functions of unbounded variation, the reader being referred to \cite{Gu,GT} for further details.

\smallskip

For  real numbers
$0 \leq a \leq b \leq T < \infty $, a vector space $V$ and an integer $k\ge 1$ we denote by
$\cac_k([a,b]; V)$ the set of functions $g : [a,b]^{k} \to V$ such
that $g_{t_1 \cdots t_{k}} = 0$
whenever $t_i = t_{i+1}$ for some $1 \leq i\le k-1$.
Such a function will be called a
\emph{$(k-1)$-increment}, and we will
set $\cac_*([a,b];V)=\cup_{k\ge 1}\cac_k([a,b];V)$. An important  operator for our purposes
is given by
\begin{equation}
  \label{eq:coboundary}
\der : \cac_k([a,b];V) \to \cac_{k+1}([a,b];V), \qquad
(\der g)_{t_1 \cdots t_{k+1}} = \sum_{i=1}^{k+1} (-1)^{k-i}
g_{t_1  \cdots \hat t_i \cdots t_{k+1}} ,
\end{equation}
where $\hat t_i$ means that this argument is omitted.
A fundamental property of $\der$
is that
$\der \der = 0$, where $\der \der$ is considered as an operator
from $\cac_k([a,b];V)$ to $\cac_{k+2}([a,b];V)$.
 We will denote $\cz\cac_k([a,b];V) = \cac_k( [a,b];V) \cap \text{Ker}\der$
and $ \cb\cac_k([a,b];V) =
\cac_k([a,b];V) \cap \text{Im}\der$.

\smallskip

Some simple examples of actions of $\der$
 are as follows: For
$g\in\cac_1([a,b];V)$, $h\in\cac_2([a,b];V)$ and $f\in\cac_3([a,b];V)$ we have
\begin{equation*}
  (\der g)_{st} = g_t - g_s,
\quad
(\der h)_{sut} = h_{st}-h_{su}-h_{ut}\quad\mbox{and}\quad
(\der f)_{suvt}=f_{uvt}-f_{svt}+f_{sut}-f_{suv}
\end{equation*}
for any $s,u,v,t\in [a,b]$.
Furthermore, it is easily checked that
$\cz \cac_{k+1}([a,b];V) =  \cb\cac_{k}([a,b];$ $V)$ for any $k\ge 1$.
In particular, the following  property holds:
\begin{lem}\label{exd}
Let $k\ge 1$ and $h\in \cz\cac_{k+1}([a,b];V)$. Then there exists a (non unique)
$f\in\cac_{k}([a,b];V)$ such that $h=\der f$.
\end{lem}

Observe that Lemma \ref{exd} implies in particular that all  elements
$h \in\cac_2([a,b];V)$  with $\der h= 0$ can be written as $h = \der f$
for some  $f \in \cac_1([a,b];V)$. Thus we have a heuristic
interpretation of $\der |_{\cac_2([a,b];V)}$:  it measures how much a
given 1-increment  differs from being an  exact increment of a
function, i.e., a finite difference.

\smallskip

Our further discussion will mainly rely on $k$-increments with $k \le 2$. We show now how to measure the size of those increments thanks to H\"older norms, and we shall thus specify a little the kind of space $V$ considered for the resolution of equation \eqref{main:spde1}:
\begin{notation}
For our future consideration, the space $V$ will be either $\R^n$, either a space of functions defined on $\R^n$ with a certain type of regularity like bounded uniformly continuous (BUC) functions, Sobolev spaces $W^{k,p}(\R^n)$ or $C^{2}$ functions. In case of function valued increments, the dependence in the space variable will often be omitted, and we will write $\vp_t$ instead of $\vp_t(\xi)$ for the value of the increment $\vp_t$ at a point $\xi\in\R^n$. The norms on our state spaces are all denoted  by $|\cdot|$.
\end{notation}

Let us now start the introduction of our H\"older type norms with 1-increments: for $f \in \cac_2([a,b];V)$ we set
\begin{align*} 
\|f\|_{\mu} =
\sup_{s,t\in [a,b]}\frac{|f_{st}|}{|t-s|^\mu},
\quad\text{and}\quad
\cac_2^\mu([a,b];V)=\lcl f \in \cac_2([a,b];V);\, \|f\|_{\mu}<\infty  \rcl.
\end{align*}
Observe now that  the usual H\"older spaces $\cac_1^\mu([a,b];V)$  are determined
        in the following way: for a continuous function $g\in\cac_1([a,b];V)$  set
\begin{equation*} 
\|g\|_{\mu}=\|\der g\|_{\mu},
\quad\text{and}\quad
\cac_1^\mu([a,b];V)=\lcl f \in \cac_1([a,b];V);\, \|f\|_{\mu}<\infty  \rcl.
\end{equation*}
Note that $\|\cdot\|_{\mu}$ is only a semi-norm on $\cac_1([a,b];V)$,
but we will  work  in general on spaces of the type
\begin{equation*}
\cac_{1,\al}^\mu([a,b];V)=
\lcl g:[a,b] \to V;\, g_a=\alpha,\, \|g\|_{\mu}<\infty \rcl,
\end{equation*}
for a given $\alpha \in V,$ on which $\|g\|_{\mu}$  is a norm.

 For $h \in \cac_3([a,b];V)$ we define in the same way
\begin{eqnarray}\label{eq:notation-norm-C3}
  \norm{h}_{\gamma,\rho} &=& \sup_{s,u,t\in [a,b] }
\frac{|h_{sut}|}{|u-s|^\gamma |t-u|^\rho}\\
\|h\|_\mu &= &
\inf \left \{\sum_i \|h_i\|_{\rho_i,\mu-\rho_i} ; (\rho_{i},h_{i})_{i \in \N} \textrm{ with } h_{i} \in \cac_3([a,b];V),
 \sum_i h_i =h , 0 < \rho_i < \mu  \right\}. \nonumber
\end{eqnarray}
Then  $\|\cdot\|_\mu$  is a norm on $\cac_3([a,b];V)$, see \cite{Gu}, and we define
$$
\cac_3^\mu([a,b];V):=\lcl h\in\cac_3([a,b];V);\, \|h\|_\mu<\infty \rcl.
$$
Eventually,
let $\cac_3^{1+}([a,b];V) = \cup_{\mu > 1} \cac_3^\mu([a,b];V)$
and  note that the same kind of norms can be considered on the
spaces $\cz \cac_3([a;b];V)$, leading to the definition of  the  spaces
$\cz \cac_3^\mu([a;b];V)$ and $\cz \cac_3^{1+}([a,b];V)$.
We shall also use a supremum norm on spaces $\cac_k$, which will be denoted by $\|\cdot\|_{\infty}$ in all cases.

\smallskip

The crucial point in this algebraic approach to the  integration of irregular
paths is that the operator
$\der$ can be inverted under mild  smoothness assumptions. This
inverse is called $\laa$. The proof of the following proposition  may be found
in \cite{Gu}, and in a simpler form in~\cite{GT}.

\medskip

\begin{prop}
\label{prop:Lambda}
There exists a unique linear map $\Lambda: \cz \cac^{1+}_3([a,b];V)
\to \cac_2^{1+}([a,b];V)$ such that
$$
\der \Lambda  = \id_{\cz \cac_3^{1+}([a,b];V)}
\quad \mbox{ and } \quad \quad
\Lambda  \der= \id_{\cac_2^{1+}([a,b];V)}.
$$
In other words, for any $h\in\cac^{1+}_3([a,b];V)$ such that $\der h=0$,
there exists a unique $g=\laa(h)\in\cac_2^{1+}([a,b];V)$ such that $\der g=h$.
Furthermore, for any $\mu > 1$,
the map $\laa$ is continuous from $\cz \cac^{\mu}_3([a,b];V)$
to $\cac_2^{\mu}([a,b];V)$ and we have
\begin{equation}\label{ineqla}
\|\Lambda h\|_{\mu} \le \frac{1}{2^\mu-2} \|h\|_{\mu} ,\qquad h \in
\cz \cac^{\mu}_3([a,b];V).
\end{equation}
\end{prop}

\medskip

This mapping  $\laa$  allows to construct a generalized Young integral:
\begin{cor}
\label{cor:integration}
For any 1-increment $g\in\cac_2 ([a,b];V)$ such that $\der g\in\cac_3^{1+}([a,b];V)$
set
$
\der f = (\id-\Lambda \der) g
$.
Then
$$
\der f_{st} = \lim_{|\Pi_{st}| \to 0} \sum_{i=0}^n g_{t_i\,t_{i+1}}
$$
for $a\leq s < t \leq b$, where the limit is taken over any partition $\Pi_{st} = \{t_0=s,\dots,
t_n=t\}$ of $[s,t]$, whose mesh tends to zero. Thus, the
1-increment $\der f$ is the indefinite integral of the 1-increment $g$.
\end{cor}

\medskip

We also need some product rules for the operator $\der$. For this
recall the following convention:
for  $g\in\cac_n([a,b];\R^{l,d})$ and $h\in\cac_m( [a,b];\R^{d,p}) $ let  $gh$
be the element of $\cac_{n+m-1}( [a,b];\R^{l,p})$ defined by
\begin{equation}\label{cvpdt}
(gh)_{t_1,\dots,t_{m+n-1}}=
g_{t_1,\dots,t_{n}} h_{t_{n},\dots,t_{m+n-1}},
\end{equation}
for $t_1,\dots,t_{m+n-1}\in [a,b].$
\begin{prop}\label{difrul} On the spaces of increments $\cac_{1},\cac_{2}$, the following relations hold true:
\begin{enumerate}
\item[{\it(i)}]
Let $g\in\cac_1([a,b];\R^{l,d})$ and $h\in\cac_1([a,b],\R^d)$. Then
$gh\in\cac_1(\R^l)$ and
\begin{equation*}
\der (gh) = \der g\,  h + g\, \der h.
\end{equation*}
\item[{\it(ii)}]
Let $g\in\cac_1([a,b]; \R^{l,d})$ and $h\in\cac_2([a,b];\R^d)$. Then
$gh\in\cac_2([a,b];\R^l)$ and
\begin{equation*}
\der (gh) = - \der g\, h + g \,\der h.
\end{equation*}
\end{enumerate}
\end{prop}

\begin{rem}
Proposition \ref{difrul} will also be applied for products $gh$ where both $g$ and $h$ are function valued increments. In this case the spatial product will be understood as a pointwise product, namely $[gh](\xi)=g(\xi) h(\xi)$.
\end{rem}

\smallskip


\subsection{Weakly controlled processes}
In this section we define the kind of generalized integral we wish to consider in order to give a meaning to equation \eqref{main:spde2}. This notion is based on the concept of weakly controlled process, which we proceed to recall.

\smallskip

Let us start with some additional notation and assumptions: first notice that throughout the paragraph we will use  both the notations $\ist f dg$ and $\cj_{st}(f\, dg)$ for the integral of a function $f$ with respect to a given function $g$ on the interval~$[s,t]$. Then recall that the basic assumption we shall use on our driving noise $x$ is the following:
\begin{hyp}\label{hyp1:x}
Let $\ga$ be a constant greater than $1/3$. The $\R^d$-valued $\ga$-H\"older path $x$ admits a L\'evy area,
 i.e. a process $\xd=\cj(dx dx)\in\cac_2^{2\ga}([0,T];\R^{d, d})$, which satisfies
$\der\xd=\der x\otimes \der x$,
that is
$$  \der\mathbf{x}^{\mathbf{2};ij}_{sut}
=
\der x^{i}_{su} \, \der x^{j}_{ut},
\quad \textrm{for all } \quad s,u,t\in\ott, \, i,j\in\{1,\ldots,d  \}.
$$
The couple $\bx=(x,\bx^{\2})$ is called rough path above $x$, and we also assume that $\bx$ is a geometrical rough path, meaning that if we denote by $\mathbf{x^{2,s}}$ the
symmetric part of $\mathbf{x^2}$ (namely $\mathbf{x^{2,s}}=\frac{1}{2}(\mathbf{x^2}+(\mathbf{x^2})^\ast)$), then the following rule holds true for $0\leq s<t\leq T$:
\begin{equation}\label{eq:geometric-rule}
\mathbf{x}^\mathbf{2,s}_{st}=\frac{1}{2} \,
\delta x_{st}\otimes\delta x_{st}.
\end{equation}
\end{hyp}

\smallskip

As mentioned above, we will be concerned with weakly controlled processes, that is paths whose increments can be expressed in a simple way in terms of the increments of $x$. More specifically:
\begin{defn}\label{def:fcp}
Let $a \leq b \leq T$ and let $z$ be a path in $\cac_1^\ka([a,b];V^{m})$ with $\ka\le\ga$ and $2\ka+\ga>1$.
We say that $z$ is a weakly controlled path based on $x$, if $\der z\in\cac_2^\ka([a,b];V^m)$ can be decomposed into
\begin{equation}\label{weak:dcp}
\der z^i=z^{x;il}\, \der x^l+ \rho^i,
\end{equation}
with $z^{x;il}\in\cac_1^\ka([a,b];V)$ and $\rho^i \in \cac_2^{2\ka}([a,b];V)$, for any $1\leq i\leq m$, $1\leq l \leq d$.
The space of weakly controlled paths on $[a,b]$ with H\"older continuity $\ka$ will be denoted by $\cq_{\bx}^{\ka}([a,b];V^m)$, and a path $z\in\cq_{\bx}^{\ka}([a,b];V^m)$ should be considered in fact as a couple $(z,z^{x})$.
The norm on $\cq_{\bx}^{\ka}([a,b];V^m)$ is given
by
\begin{equation*}
\cn[z;\cq_{\bx}^{\ka}([a,b];V^m)]
= \|z\|_{\ka} + \|z^{x}\|_{\infty} + \|z^{x}\|_{\ka} + \|\rho\|_{2\ka}.
\end{equation*}
\end{defn}

Let us first see how smooth functions depending on time act on weakly controlled paths, in a proposition which is a mere variation of \cite{Gu}:
\begin{prop}\label{cp:weak-phi}
Let $z\in\cq_{\bx}^{\ka}([a,b];V^{m_1})$ with decomposition (\ref{weak:dcp}) and initial condition $z_0=\al$, and
$G \in   C_{b}^{\tau,2}([a,b]\times\R^{m_1};\R^{m_2})$ for $\tau>0$. Set $\hat z_t=G(t,z_t)$, with initial condition $\hal=G(a,\al)$. Then the increments of $\hat z$ can be decomposed into
$$
\der \hat z_{st}= \hat z_{s}^{x} \, \der x_{st} +\hro + \der G(\cdot,z_t)_{st}
$$
with
$$
\hat z^{x}_s= \nabla_z G(s,z_s)z^x_s
\quad\mbox{ and }\quad
\hro_{st}= \nabla_z G(s,z_s)\rho_{st} + \lc (\der G(s,z))_{st}-\nabla_z G(s,z_s)(\der z)_{st} \rc.
$$
Furthermore $\hat z \in \cq_{\bx}^{\ka}([a,b];V^{m_2})$, and
\begin{equation}\label{bnd:phi}
\cn[\hat z;\cq_{\bx}^{\ka}([a,b];V^{m_2})]\le
c_{G}\lp 1+\cn^2[z;\cq_{\bx}^{\ka}([a,b];V^{m_1})]  \rp.
\end{equation}
\end{prop}

We now turn  to a proposition giving an expression for the integration of weakly controlled paths with respect to $x$, when an additional term involving Lebesgue integrals is present. This is again a mere variation on the results contained in \cite{Gu,GT}, for which we introduce first some additional notation:

\begin{notation}
Let $\eta$ be a path in $\cac_{1}^{0}([a,b];V)$, where $\cac_{1}^{0}([a,b];V)$ stands for the space of $V$-valued continuous and bounded functions defined on $[a,b]$. For all $a\le s\le t \le b$ we will write $\ci_{st}(\eta)$ for the Lebesgue integral $\int_{s}^{t} \eta_{u} \, du$. We also write  $\Id$ for the identity function on $\R$, whose increments are simply given by $\der\Id_{st}=t-s$. We thus trivially have $\ci_{st}(\eta) = \eta_{s}\der\Id_{st}+ o(|t-s|)$.
\end{notation}

\begin{prop}\label{intg:wcdx}
For a given $\ga>1/3$ and $\ka<\ga$,
let $x$ be a process satisfying Hypothesis~\ref{hyp1:x}. Furthermore,  let
$\eta\in \cac_{1}^{0}([a,b];V)$, and $\mu\in\cq_{\bx}^{\ka}([a,b];V^d)$ whose increments can be decomposed as
\begin{align}
\der \mu^l&=\mu^{x;lk} \, \der x^k +\rho^l,
\quad\mbox{ where }\quad
\mu^x\in\cac_1^\ka([a,b];V^{d,d}), \, \rho\in\cac_2^{2\ka}([a,b];V^d).  \label{dcp:z1}
\end{align}
Define $z$ by $z_a=\al\in V$ and
\begin{equation}\label{dcp:wcdx}
\der z=\mu^l \, \der x^l  +\mu^{x;lk} \, \mathbf{x}^{\mathbf{2};kl}
+ \ci(\eta) + \laa(\rho^l\, \der x^l+\der \mu^{x;lk}\mathbf{x}^{\mathbf{2};kl}).
\end{equation}
Then:

\smallskip

\noindent
\emph{(i)}
$z$ is well-defined as an element of $\cq_{\bx}^{\ka}([a,b];V)$, and $\der z_{st}$ coincides with the Lebesgue-Stieljes  integral $\int_{s}^{t} \mu^{i}_{u} \, dx^{i}_{u} + \int_{s}^{t} \eta_{u} \, du$ whenever $x$ is a differentiable function.

\smallskip

\noindent
\emph{(ii)}
The semi-norm of $z$ in $\cq_{\bx}^{\ka}([a,b];V)$ can be estimated as
\begin{align}\label{bnd:norm-iwcdx}
&\cn[z;\cq_{\bx}^{\ka}([a,b];V)]\\
&\le
c_{x}
\lp 1 +\|\bet\| + (b-a)^{\ga-\ka}\lcl\|\bet\|+\cn[\mu;\cq_{\ka,\bet}([a,b];V^d)]+\cn[\eta;\cac_1^0([a,b];V)]\rcl\rp, \notag
\end{align}
where $\bet\equiv\mu_a$, and where the constant $c_{x}$ can be bounded as follows: $c_x\le c [ |x|_{\ga}+|\xd|_{2\ga}]$, for a universal constant $c$.

\smallskip

\noindent
\emph{(iii)}
It holds
\begin{equation}\label{rsums:iwcdx}
\der z_{st}
=\lim_{|\Pi_{st}|\to 0}\sum_{q=0}^n
\lc \mu^l_{t_{q}} \, \der x^l_{t_{q}, t_{q+1}}+\eta_{t_{q}} \delta\Id_{t_{q}, t_{q+1}}
+ \mu^{x;lk}_{t_{q}} \mathbf{x}^{\mathbf{2};kl}_{t_{q}, t_{q+1}} \rc
\end{equation} for any $a\le s<t\le b$,
where the limit is taken over all partitions
$\Pi_{st} = \{s=t_0,\dots,t_n=t\}$
of $[s,t]$, as the mesh of the partition goes to zero.
\end{prop}

\begin{rem}\label{rmk:diff-controlled-proc}
A lot of the considerations below will be based on differentiation arguments under $\int$ signs. They can be justified easily in natural situations e.g. when $\mu$ is a controlled process taking values in a space of differentiable functions of the form $V=W^{k,p}$ with $k\ge 1$. In this context, let $z$ be defined by  \eqref{dcp:wcdx}. Then $Dz$ verifies:
\begin{equation}
\der Dz=D\mu^l \, \der x^l+D\mu^{x;lk} \, \mathbf{x}^{\mathbf{2};kl}
+ \ci\lp D\eta \rp
+\laa\lp D\rho^l\, \der x^l+\der D\mu^{x;lk}\mathbf{x}^{\mathbf{2};kl} \rp.\label{dcp:Dwcdx}
\end{equation}
\end{rem}

With these considerations in hand we get in particular a space-time expansion for the process $z$ in the case of a space $V$ of differentiable functions.

\begin{cor}\label{taylor:wcdx}
Assume the hypothesis of Proposition \ref{intg:wcdx} hold true, and that $V=\cac^{2}$. Then
\begin{multline*}
z_t(\nu)-z_s(\te)=\eta_s(\te)\, \delta\Id_{st}+ \mu^l_s(\te)\, \der x^l_{st}+\mu^{x;lk}_s(\te)\,\mathbf{x}^{\mathbf{2};kl}_{st}\\
+Dz_s(\te)\lc\nu-\te\rc+D\mu^l_s(\te) \, \der x^l_{st} \lc\nu-\te\rc
+\crr\lp(s,\te),(t,\nu)\rp,
\end{multline*}
for all $s,t\in[a,b]$ and any $\te,\nu\in \R^n$, where
\begin{equation*}
\lln\crr\lp(s,\te),(t,\nu)\rp\rrn\approx o\lp|t-s|+|\nu-\te|\rp.
\end{equation*}
\end{cor}

\begin{proof}
We shall prove the expansion in the case $a\leq s<t\leq b$, the proof in the case $a\leq t<s\leq b$ being very similar.
Fix then $a\leq s<t\leq b$ and $\te,\nu\in \R^n$. We write
\begin{equation}\label{taylor2v}
z_t(\nu)-z_s(\te)=z_t(\nu)-z_t(\te)+z_t(\te)-z_s(\te).
\end{equation}
The time increment $z_t(\te)-z_s(\te)$ can be treated thanks to Proposition \ref{intg:wcdx} and we get
\begin{equation}\label{taylorrp1}
z_t(\te)-z_s(\te)=\mu^l_s(\te)\, \der x^l_{st}+\eta_s(\te)\, \der \Id_{st}
+\mu^{x;lk}_s(\te)\,\mathbf{x}^{\mathbf{2};kl}_{st}+\crr_1\lp(s,\te),(t,\nu)\rp,
\end{equation}
where
\begin{equation*}
\crr_1\lp(s,\te),(t,\nu)\rp=
\laa_{st}\lp \rho^l(\te)\, \der x^l+\der \mu^{x;lk}(\te)\,\mathbf{x}^{\mathbf{2};kl}\rp +o(|t-s|),
\end{equation*}
which yields
\begin{equation*}
\lln\crr_1\lp(s,\te),(t,\nu)\rp\rrn\approx o(|t-s|).
\end{equation*}

Going back to \eqref{taylor2v}, the space increment $z_t(\nu)-z_t(\te)$ is now handled with usual Taylor arguments:
\begin{equation}\label{taylorusual}
z_t(\nu)-z_t(\te)=Dz_t(\te)[\nu-\te]+\frac{1}{2}\,D^2z_t(\te+\la (\nu-\te))[\nu-\te]\otimes[\nu-\te],
\end{equation}
where $\la\in ]0,1[$ is a given value. Invoking now relation \eqref{dcp:Dwcdx} we can write (\ref{taylorusual}) in the following form:
\begin{align}\label{taylorrp2}
z_t(\nu)-z_t(\te)&=Dz_s(\te)[\nu-\te]+(\der Dz(\te))_{st}[\nu-\te]+\frac{1}{2}\,D^2z_t(\te+\la (\nu-\te))[\nu-\te]\otimes[\nu-\te]\notag\\
&=Dz_s(\te)[\nu-\te]+D\mu^l_s(\te)\,(\der x^l)_{st}[\nu-\te]+\crr_2\lp(s,\te),(t,\nu)\rp,
\end{align}
where $\crr_2((s,\te),(t,\nu))$ can be decomposed as
\begin{multline*}
D\eta_s(\te)\, \der \Id_{st}[\nu-\te]+D\mu^{x;lk}_s(\te)\,\mathbf{x}^{\mathbf{2};kl}_{st}[\nu-\te]
+\frac{1}{2}\,D^2z_t(\te+\la (\nu-\te))[\nu-\te]\otimes[\nu-\te] \\
+\lcl\laa_{st}\lp D\rho^l(\te)\, \der x^l+\der D\mu^{x;lk}(\te)\,\mathbf{x}^{\mathbf{2};kl}\rp \rcl[\nu-\te] +o(|t-s|)[\nu-\te].
\end{multline*}
We thus easily get the following estimate:
\begin{equation}\label{eq:bnd-R2}
\lln\crr_2\lp(s,\te),(t,\nu)\rp\rrn \approx o\lp|t-s|+|\nu-\te|\rp,
\end{equation}
where we notice that the term $\mathbf{x}^{\mathbf{2};kl}_{st}[\nu-\te]$ is bounded thanks to the elementary inequality $ab\le \frac12(a^{2}+b^{2})$, valid for all $a,b\in\R$.
Let us now plug the decompositions (\ref{taylorrp1}) and (\ref{taylorrp2}) into expression
(\ref{taylor2v}). This finishes the proof of our claim.
\end{proof}

\subsection{Strongly controlled processes}
In order to identify solutions to equations of type~\eqref{main:spde2} we shall include the drift term into the definition of controlled processes, and thus push forward our expansion with respect to the increments of our rough signal $x$. This yields the following definition:
\begin{defn}\label{def:strongly-ctld-process}
Let $a \leq b \leq T$ and let $z$ be a path in $\cac_1^\ka([a,b];V^{m})$ with $\ka\le\ga$ and $2\ka+\ga>1$.
We say that $z$ is a strongly controlled path based on $x$, if $\der z\in\cac_2^\ka([a,b];V^m)$ can be decomposed into
\begin{equation}\label{eq:strg-ctld-dcp}
\der z^i=z^{x;il}\, \der x^l+ z^{xx;ilk}\, \bx^{\2;kl} + \ci\lp z^{t;i} \rp +
z^{\flat;i},
\end{equation}
with $z^{x;il},z^{xx;ilk}\in\cac_1^\ka([a,b];V)$, $z^{t;i}\in\cac_1^{0}([a,b];V)$, $z^{\flat;i} \in \cac_2^{3\ka}([a,b];V)$, for any $1\leq i\leq m$, $1\leq l \leq d$ and where we assume the further relation between $z^{x}$ and $z^{xx}$:
\begin{equation}\label{eq:rel-zx-zxx}
\der z^{x;il} = z^{xx;ilk}\, \der x^{k} +  z^{\sharp;il},
\quad\text{with}\quad
z^{\sharp;il} \in \cac_2^{2\ka}([a,b];V).
\end{equation}
The space of strongly controlled paths on $[a,b]$ with H\"older continuity $\ka$ will be denoted by $\cs_{\bx}^{\ka}([a,b];V^m)$, and a path $z\in\cs_{\bx}^{\ka}([a,b];V^m)$ has to be considered as a vector $(z,z^{x},z^{xx},z^{t})$. The norm on $\cs_{\bx}^{\ka}([a,b];V^m)$ is given by
\begin{equation*}
\cn[z;\cs_{\bx}^{\ka}([a,b];V^m)]
= \|z\|_{\ka} + \|z^{x}\|_{\infty} + \|z^{x}\|_{\ka} + \|z^{xx}\|_{\infty} + \|z^{xx}\|_{\ka}
+ \|z^{t}\|_{\infty}
+ \|z^{\flat}\|_{3\ka}
+ \|z^{\sharp}\|_{2\ka}.
\end{equation*}
\end{defn}

\begin{rem}
The definition of strongly controlled processes in case of a driving noise $x\in\cac^{\ga}$ with $1/3<\ga\le 1/2$ corresponds in fact to the definition of weakly controlled processes corresponding to $x\in\cac^{\ga}$ with $1/4<\ga\le 1/3$. Further information on these structures can be found in \cite{CHLT,Gu10,TT}.
\end{rem}

\smallskip

The space of $V^{m}$-valued strongly controlled processes is obviously a subset of the weakly controlled processes $\cq_{\bx}^{\ka}([a,b];V^m)$. It turns out that this subset has a degenerate structure, in the sense that strongly controlled processes are in fact integrals of weakly controlled processes with respect to $x$. This is the meaning of the following proposition:
\begin{prop}\label{prop:strg-proc-as-integrals}
Let $z$ be a strongly controlled process as in Definition \ref{def:strongly-ctld-process}. Then

\smallskip

\noindent
\emph{(i)} The increment $z^{\flat}$ is uniquely determined by the other components of the strongly controlled process $z$.

\smallskip

\noindent
\emph{(ii)}
for all $s,t\in\ott$ we have
\begin{equation*}
\der z^{i}_{st} = \ist z^{x;il}_{u} \, dx^{l}_{u} + \ist z^{t;i}_{u} \, du,
\end{equation*}
where the integral with respect to $x$ is understood in the sense of Proposition \ref{intg:wcdx}.
\end{prop}

\begin{proof}
Compute $\der z^{\flat;i}$. According to equation \eqref{eq:strg-ctld-dcp} and Proposition \ref{difrul}, plus the fact that $\der\ci(\eta)=0$, we have
\begin{eqnarray*}
 \der z^{\flat;i}&=&
 \der \lc z^{i}-z^{x;il}\,\der x^{l}- z^{xx;ilk} \,\bx^{\2;lk}  \rc  \\
& =& \der z^{x;il}\,\der x^{l} + \der z^{xx;ilk} \,\bx^{\2;kl} - z^{xx;ilk} \,\der \bx^{\2;kl}.
\end{eqnarray*}
Invoking now relation \eqref{eq:rel-zx-zxx} and Hypothesis \ref{hyp1:x} we obtain
\begin{eqnarray*}
\der z^{\flat;i} &=& \lp z^{xx;ilk}\, \der x^{k} +  z^{\sharp;il} \rp \der x^{l} + \der z^{xx;ilk} \,\bx^{\2;kl}
- z^{xx;ilk} \,\der x^{k} \, \der x^{l}  \\
&=& z^{\sharp;il}  \der x^{l} + \der z^{xx;ilk} \,\bx^{\2;kl}  ,
\end{eqnarray*}
and it is thus easily seen that $\der z^{\flat;i}\in\cac_{3}^{3\ka}$. Owing to Proposition \ref{prop:Lambda}, we can thus write
\begin{equation*}
z^{\flat;i} = \laa\lp   \der z^{xx;ilk} \,\bx^{\2;kl}   +z^{\sharp;il}  \der x^{l} \rp,
\end{equation*}
which yields our first assertion.

\smallskip

In order to prove our claim (ii), it is thus sufficient to observe that the decomposition of $\der z^{i}$ and formula \eqref{dcp:wcdx} coincide up to increments of order $3\ka>1$.

\end{proof}

\smallskip

The following proposition proves that the composition of two function-valued strongly controlled processes is still a strongly controlled process. This fact will be used in order to define transformations of test functions and jets for viscosity solutions to rough PDEs.

\begin{prop}
\label{prop:compo-scp}
Let $x$ be a path satisfying Hypothesis \ref{hyp1:x}, and consider two $\cac^{3}(\R^m;\R^{m})$-valued strongly controlled processes $y, z$ admitting a decomposition of type \eqref{eq:strg-ctld-dcp}. Then $z\circ y$ is another strongly controlled process with decomposition
\begin{equation*}
[z\circ y]^{x;il_1}(\te)= z^{x;i l_1}(y(\te))+\partial_{\te^j}z^{i}(y(\te))\,y^{x;jl_1}(\te),
\end{equation*}
\begin{align*}
[z\circ y]^{xx;il_1l_2}(\te)&=
z^{xx;il_1l_2}(y(\te)) +\partial_{\te^j}z^{i}(y(\te))\,y^{xx;jl_1l_2}(\te)\\
&\quad+\partial_{\te^j}z^{x;il_2}(y(\te))\,y^{x;jl_1}(\te)+ \partial_{\te^j}z^{x;il_1}(y(\te))\,y^{x;jl_2}(\te)\\
&\quad+\partial^2_{\te^{j_1}\te^{j_2}}z^i(y(\te))\,y^{x;j_1l_1}(\te)\,y^{x;j_2l_2}(\te),
\end{align*}
and
\begin{equation*}
[z\circ y]^{t;i}(\te)=z^{t;i}(y(\te))+\partial_{\te^j}z^i(y(\te))\,y^{t;j}(\te).
\end{equation*}
\end{prop}
\begin{proof}
Write
\begin{equation*}
z_t(y_t(\te)) - z_s(y_s(\te))
= \lc z_t(y_t(\te)) - z_t(y_s(\te)) \rc + \lc z_t(y_s(\te)) - z_s(y_s(\te)) \rc
\equiv S_{st}(\te) + T_{st}(\te),
\end{equation*}
and deal with the two terms $S_{st}(\te)$ and $T_{st}(\te)$ separately. Note that in the remainder of the proof we shall denote by $r^\flat$ any remainder of order $\cac^{1+}$, independently of its particular value.

\smallskip

In order to handle $S_{st}(\te)$, we invoke the fact that $z$ can be differentiated in the space variable plus Remark \ref{rmk:diff-controlled-proc}, which yield:
\begin{equation}\label{eq:gen-dcp-Sst-1}
S_{st}^{i}(\te)
= \partial_{\te^j} z_t^{i}(y_s(\te)) \, \der y_{st}^{j}(\te)
+\frac12 \partial_{\te^{j_1}\te^{j_2}}^{2} z_t^{i}(y_s(\te)) \, \der y_{st}^{j_1}(\te) \, \der y_{st}^{j_2}(\te)
+r_{st}^{\flat}.
\end{equation}
We now plug the decomposition
\begin{equation*}
\der y^j_{st}(\te) = y_s^{x;jl_1}(\te) \, \der x_{st}^{l_1}
+ y_s^{xx;j l_1 l_2}(\te) \, \bx_{st}^{\2;l_2l_1}+ \ci_{st}\lp  y^{t;j}(\te)\rp
+ r_{st}^{\flat}(\te)
\end{equation*}
into \eqref{eq:gen-dcp-Sst-1}, replace $\partial_{\te^j} z_t^{i}(y_s(\te))$ by $\partial_{\te^j} z_s^{i}(y_s(\te))$ and $\partial_{\te^j\te^k}^{2} z_t^{i}(y_s(\te))$ by $\partial_{\te^j\te^k}^{2}z_s^{i}(y_s(\te))$ and only keep track of terms whose H\"older regularity is $<1$, which gives
\begin{multline*}
S_{st}^{i}(\te)= \partial_{\te^j} z_s^{i}(y_s(\te)) \, y^{x;jl_1}(\te) \, \der x_{st}^{l_1}
+\frac12 \partial_{\te^{j_1}\te^{j_2}}^{2}  z_{s}^{i}(y_s(\te)) \, y^{x;j_1l_1}(\te)y^{x;j_2l_2}(\te)
\, \der x_{st}^{l_1}  \, \der x_{st}^{l_2} \\
+ \partial_{\te^j}\, \der z_{st}^{i}(y_s(\te)) \, y^{x;jl_1}(\te) \, \der x_{st}^{l_1}
+ \partial_{\te^j} z_s^{i}(y_s(\te)) \, y_s ^{xx;j l_1 l_2}(\te)\, \bx_{st}^{\2;l_2l_1}
+\ci_{st}\lp \partial_{\te^j} z^{i}(y(\te)) \,y^{t;j}(\te) \rp
+ r_{st}^{\flat}.
\end{multline*}
Next resort to the decomposition of $\der \partial_{\te^j} z$ given by \eqref{eq:strg-ctld-dcp} and also invoke the geometric assumption \eqref{eq:geometric-rule} for the term $\der x_{st}^{l_1}  \, \der x_{st}^{l_2}$. As the reader might easily check, this yields $S^{x;il_1}(\te)=  \partial_{\te^j}z^{i}(y(\te))\,y^{x;jl_1}(\te)$, $S^{t;i}(\te)= \partial_{\te^j}z^i(y(\te))\,y^{t;j}(\te)$ and
\begin{multline*}
S^{xx;il_1l_2}(\te)=
\partial_{\te^j}z^{i}(y(\te))\,y^{xx;jl_1l_2}(\te)+\partial_{\te^j}z^{x;il_2}(y(\te))\,y^{x;jl_1}(\te)+ \partial_{\te^j}z^{x;il_1}(y(\te))\,y^{x;jl_2}(\te)\\
+\partial^2_{\te^{j_1}\te^{j_2}}z^i(y(\te))\,y^{x;j_1l_1}(\te)\,y^{x;j_2l_2}(\te).
\end{multline*}
The term $T_{st}(\te)$ gives the remaining contributions of our claim, by just composing relation~\eqref{eq:strg-ctld-dcp} with $y$. It is also easy to see that relation ~\eqref{eq:rel-zx-zxx} between $[z\circ y]^{x}$ and $[z\circ y]^{xx}$ holds.
\end{proof}

The previous proposition yields a space-time expansion for the composition $z\circ y$ of strongly controlled processes, which will be used in the definition of jets.

\begin{cor}
\label{cor:compo-scp}
Assume the same hypotheses as in Proposition \ref{prop:compo-scp}. Then
\begin{align*}
[z\circ y]^i_t(\nu)-[z\circ y]^i_s(\te)&=[z\circ y]^{t;i}(\te)\, \delta\Id_{st}+ [z\circ y]^{x;il_1}(\te)\, \der x^l_{st}+[z\circ y]^{xx;il_1l_2}(\te)\,\mathbf{x}^{\mathbf{2};kl}_{st}\\
&\quad+D[z\circ y]^i_s(\te)\lc\nu-\te\rc+D[z\circ y]^{x;il_1}(\te) \, \der x^l_{st} \lc\nu-\te\rc\\
&\quad+o\lp|t-s|+|\nu-\te|\rp,
\end{align*}
for all $1\leq i\leq n$, $s,t\in[a,b]$ and any $\te,\nu\in \R^n$.
\end{cor}
\begin{proof}
It is a direct application of Proposition \ref{prop:compo-scp} and Corollary \ref{taylor:wcdx}.
\end{proof}

\subsection{Preliminary results on rough flows}\label{sec:prelim-stoch-flows}

The existence of the rough flow $\phi$ generated by a family of vector fields ${A}_{1},\ldots,{A}_{d}$ is a well-known result, which can be stated as follows:

\begin{prop}\label{prop:dcp-flow}
Let $x$ be a path satisfying Hypothesis \ref{hyp1:x}. For $\eta\in\R^{m}$ and a family of smooth and bounded vector fields ${A}_{1},\ldots,{A}_{d}$, consider the unique solution $\phi_t(\te)$ to the following rough equation:
\begin{equation}\label{eq:rde-flow}
\phi_t(\eta) = \eta + \sum_{l=1}^{d} \int_0^t{A}_{l}(\phi_s(\eta)) \, dx_{s}^{l}.
\end{equation}
Then:

\smallskip

\noindent
\emph{(i)} For any $t\in\R_+$, $\phi_t$ is a $\cac^3$-diffeomorphism.

\smallskip

\noindent
\emph{(ii)} $\phi(\eta)$ is well defined as a strongly controlled process in $\cs_\bx^\ga(\R^m)$, with decomposition:
\begin{equation*}
\phi^{x;il_1} = {A}^i_{l_1} \circ \phi, \qquad  \phi^{xx;il_1l_2} = [{A}_{l_2} {A}^i_{l_1}]\circ\phi,
\qquad \phi^{t;i}=0,
\end{equation*}
where we have used the interpretation of our vector fields $A_l$ as derivative operators. Namely, for a smooth function $g$ defined on $\R^{m}$ we set $A_{l}g=A_{l}^{k}\partial_{\te^k}g$ and in particular ${A}_{l_2} {A}^i_{l_1}(\te)=A_{l_2}^{k}(\te)\partial_{\te^k}{A}^i_{l_1}(\te)$.
\end{prop}

We now give the equivalent of Proposition \ref{prop:dcp-flow} for the inverse of the flow $\phi$.

\begin{prop}\label{prop:dcp-invflow}
Let $x$ be a path satisfying Hypothesis \ref{hyp1:x}. For $\eta\in\R^{m}$ and the family of vector fields ${A}_{1},\ldots,{A}_{d}$ of Proposition \ref{prop:dcp-flow}, consider the unique solution $\zeta_t(\eta)$ to the following rough equation:
\begin{equation}\label{eq:rde-invflow}
{\zeta}_t(\eta)=\eta-\sum_{l_1=1}^d \int_0^t {A}_{l_1}{\zeta}_s(\eta)\,dx_s^{l_1},
\end{equation}
where ${A}_{l_1}{\zeta}_s$ stands for ${A}_{l_1}^{i} \partial_{\te_{i}}\zeta_s$.
Then:

\smallskip

\noindent
\emph{(i)} ${\zeta}(\eta)$ is well defined as a strongly controlled process in $\cs_\bx^\ga(\R^m)$, with decomposition:
\begin{equation*}
{\zeta}^{x;il_1}= -{A}_{l_1}{\zeta}^i, \qquad  {\zeta}^{xx;il_1l_2} = {A}_{l_1} {A}_{l_2}{\zeta}^i,
\qquad \zeta^{t;i}=0,
\end{equation*}
where we have used the interpretation of our vector fields $A_l$ as derivative operators and
\begin{equation*}
{A}_{l_1} {A}_{l_2} \, g
={A}_{l_1}^{j_1}\,\partial_{\te^{j_1}}{A}^{j_2}_{l_2}\,\partial_{\te^{j_2}} g
+ {A}_{l_1}^{j_1} {A}_{l_2}^{j_2} \, \partial^2_{\te^{j_1}\te^{j_2}} g.
\end{equation*}

\smallskip

\noindent
\emph{(ii)} For all $t\in[0,T]$ we have $\zeta_{t}=\phi^{-1}_{t}$, where  $\phi$ is the solution to \eqref{eq:rde-flow}.
\end{prop}

\begin{proof}
Item (i) can be thought of as an easy exercise, while item (ii) can be proven exactly as in \cite{Ku-Flour}.
\end{proof}

\section{Viscosity solutions to rough PDEs}\label{sec:viscosity-solutions}
We now go back to equation \eqref{main:spde1} in quite a general context. However, for sake of clarity, we shall restrict our analysis to the case of a dependence of $F$ and $\si$ on $u$ and $Du$ only. Namely, we consider a general equation of the form:
\begin{equation}\label{eq:spde-integral}
u_t(\te)=\alpha(\te)+\int_0^t F\lp u_r(\te),Du_r(\te)\rp dr
+\sum_{l=1}^d\int_0^t\si^l\lp u_r(\te), Du_r(\te)\rp dx_r^l,
\end{equation}
In this section we shall clarify what we mean by strong and viscosity solution to equation~\eqref{eq:spde-integral}.

\subsection{Strong solutions to rough PDEs}
There are several ways to define solutions to rough PDEs driven by a finite dimensional noise (see \cite{CFO,GT} among others). For our purposes and with our previous notations in mind, the most natural way to do so is the following:
\begin{defn}\label{def:clasol}
Let $u$ be a weakly controlled process in $\cq_\bx^\ka([0,T];V)$ with $V=\cac^2(\R^n)$ and $\ka>1/3$, and consider a function $\si$ in $\cac^{2}(\R\times \R^{n})$.
We say that $u$ is a strong solution to the rough PDE \eqref{eq:spde-integral} if $u$ is in fact a strongly controlled process in $\cs_\bx^\ka([0,T];V)$, such that
\begin{equation*}
u^{x;l} = \si^l(u,Du), \quad u^{t}= F(u,Du)
\end{equation*}
and
\begin{equation}\label{eq:dcp-uxx-strong-solt}
u^{xx;l_1l_2} = \partial_u\si^{l_1}(u,Du)\,\si^{l_2}(u,Du)
+ \partial_{p^j}\si^{l_1}(u,Du)\,\partial_{\te^j} [ \si^{l_2}(u,Du)],
\end{equation}
where in the last expression, the notation $\partial_{\te^j} [ \si^k(u,Du)]$ is a compact expression for the quantity
\begin{equation}\label{eq:expression-Dtheta-sigma-u}
\partial_{\te^j} [ \si^{l_2}(u,Du)]
=
\partial_u\si^{l_2}(u,Du)\,\partial_{\te^j} u +\partial_{p^i}\si^{l_2}(u,Du)\,\partial_{\te^i\te^j}u.
\end{equation}
\end{defn}

\begin{proof}[Heuristics for Definition \ref{def:clasol}]
Let us denote by $\rho^{\beta}$ any generic remainder with H\"older regularity $\beta$. Then the weakly controlled structure of the candidate solution $u$ is given by:
\begin{equation}\label{eq:u-weak-ctrled}
\der u_{st} = \si^{l_1}(u_s,Du_s) \, \der x^{l_1}_{st} + \rho_{st}^{2\ga}
\equiv u^{x;l_1}_{s} \, \der x^{l_1}_{st} + \rho_{st}^{2\ga}.
\end{equation}
We now expand further the $2\ga$-H\"older term in $\rho_{st}^{2\ga}$: it is given by
\begin{equation*}
\int_{s}^{t} K_{sr}^{l_1} \, dx_{r}^{l_1},
\quad\mbox{with}\quad
K_{sr}^{l_1} = \si^{l_1}(u_r,Du_r) - \si^{l_1}(u_s,Du_s).
\end{equation*}
Furthermore, equation \eqref{eq:u-weak-ctrled} yields
\begin{equation*}
\der u_{st}= \si^{l_2}(u_s,Du_s) \, \der x^{l_2}_{st} + \rho_{st}^{2\ga},
\quad\mbox{and}\quad
\der D^{l_3}u_{st}= \partial_{\te_{l_3}}\si^{l_2}(u_s,Du_s) \, \der x^{l_2}_{st} + \rho_{st}^{2\ga}.
\end{equation*}
Plugging these relations into the definition of $K$ we get
\begin{equation*}
K_{sr}^{l_1} =
\lc \partial_{u} \si^{l_1}(u_s,Du_s) \, \si^{l_2}(u_s,Du_s)
+ \partial_{p_{l_3}} \si^{l_1}(u_s,Du_s) \, \partial_{\te_{l_3}}\si^{l_2}(u_s,Du_s) \rc \der x^{l_2}_{st} + \rho_{st}^{2\ga},
\end{equation*}
and thus $\int_{s}^{t} K_{sr}^{l_1} \, dx_{r}^{l_1}= u^{xx;l_1l_2} \bx^{\2;l_2l_1}+\rho_{st}^{3\ga}$, where
\begin{equation*}
u^{xx;l_1l_2}_{s} = \partial_{u} \si^{l_1}(u_s,Du_s) \, \si^{l_2}(u_s,Du_s)
+ \partial_{p_{l_3}} \si^{l_1}(u_s,Du_s) \, \partial_{\te_{l_3}}\si^{l_2}(u_s,Du_s),
\end{equation*}
which is our decomposition \eqref{eq:dcp-uxx-strong-solt}.

\end{proof}

\begin{rem}
Note that even if we consider differential equations of the first order in $u$
  it is clear from the rough formulation (and specifically equation \eqref{eq:expression-Dtheta-sigma-u}) that in general, the equation is
  of second order. This has links with the well-known phenomenon of
  super-parabolicity of SPDEs.
\end{rem}

We will use the a priori structure of strong solutions to guess the natural form of test functions for the viscosity formulation of our rough PDE. To this aim we first get a space-time expansion for strong solutions to equation \eqref{eq:spde-integral}:

\begin{prop}
\label{class-decom}
Suppose that
$$
F\in C^{3}\lp\R\times\R^n;\R\rp \text{ and } \si\in C^{3}\lp\R\times \R^n;\R^d\rp.
$$
Let $u$ be a strong solution to \eqref{eq:spde-integral}. Then, for any $0\leq s,t\leq T$ and any $\te,\nu\in \R^n$, it holds that
\begin{multline}\label{tx:taylor}
u_t(\nu)=u_s(\te)+a_s(\te)\, \der\Id_{st} + b_{s}^l(\te)\, \der x^l_{st}
+c_{s}^{lk}(\te)\,\mathbf{x}^{\mathbf{2},kl}_{st} \\
+p_{s}^j(\te)\,[\nu^j-\te^j]+q_{s}^{jl}(\te)\,\der x^l_{st}\, [\nu^j-\te^j]
+o\lp|t-s|+|\nu-\te|\rp,
\end{multline}
where
\begin{align*}
&a_s(\te)=F(u_s(\te),Du_s(\te)), \qquad b_{s}^l(\te) =\si^l(u_s(\te),Du_s(\te)),  \\
&p_{s}^j(\te)=\partial_{\te^j} u_s(\te),  \qquad
q_{s}^{jl}(\te) =\partial_{\te^j}[\si^l(u_s(\te),Du_s(\te))],\\
&c_{s}^{lk}(\te)=b_{s}^k\,\partial_u\si^l(u_s(\te),Du_s(\te))+ q_{s}^{jk}(\te)\,\partial_{p^j}\si^l(u_s(\te),Du_s(\te)).
\end{align*}
\end{prop}

\smallskip

\begin{proof}
We apply Proposition \ref{cp:weak-phi} to $z=(u,Du)$ and $G=\si$ and then Corollary~\ref{taylor:wcdx} to the weakly controlled process defined by $\si(u,Du)$ and  $\eta=F(u,Du)$.
\end{proof}

\subsection{Viscosity solutions}
As mentioned above, the structure of test functions for our viscosity solutions is based on the a priori structure of strong solution. Specifically, the set of test functions related to a rough coefficient $\si$ belongs to a subset of strongly controlled processes defined as follows:

\begin{defn}\label{def:test-fct}
Consider a coefficient $\si\in\cac^{3}(\R\times \R^n;\R^d)$.
Let $\vp$ be a strongly controlled process in $\cs_\bx^\ka([0,T];V)$ with $V=\cac^2(\R^n)$ and $\ka>1/3$.
We say that $\psi$ is an element of the test functions $\ct_{\si}$ if the coefficients in the decomposition \eqref{eq:strg-ctld-dcp} of $\psi$ satisfy $\psi^{x;l} = \si^l(\psi,D\psi)$ and
\begin{equation}\label{eq:dcp-test-fct}
\psi^{xx;l_1l_2} = \partial_u\si^{l_1}(\psi,D\psi)\,\si^{l_2}(\psi,D\psi)
+ \partial_{p^j}\si^{l_1}(\psi,D\psi)\,\partial_{\te^j} [ \si^{l_2}(\psi,D\psi)],
\end{equation}
where we have used the convention \eqref{eq:expression-Dtheta-sigma-u} for the definition of $\partial_{\te^j} [ \si^{l_1}(\psi,D\psi)]$.
\end{defn}

\begin{rem}\label{rmk:nonempty-T(sigma)}
Owing to Proposition \ref{prop:strg-proc-as-integrals}, a generic element of $\ct_{\si}$ satisfies the relation
\begin{equation*}
\der\psi_{st} = \int_{s}^{t} \si^l(\psi_u,D\psi_u) \, dx_{u}^{l} + \int_{s}^{t} \psi^{t}_{u} \, du,
\end{equation*}
for a given continuous element $\psi^{t}$. It is thus not obvious a priori that $\ct_{\si}$ is a nonempty set, not to mention the fact that it contains enough functions to fully characterize solutions to equation \eqref{eq:spde-integral}. However, this kind of property will be observed  on the particular examples of equations treated in the next sections.
\end{rem}

Similarly to what is done in \cite{LSo2}, we now define the concept of viscosity solution to our equation \eqref{eq:spde-integral} in the following manner:
\begin{defn}
\label{def:viscosity-critical-points}
A function $u\in \cac([0,T]\times \R^n)$ is called a viscosity subsolution (resp. supersolution) of \eqref{eq:spde-integral} if
\begin{itemize}
\item[(i)] We have $u(0,\te)\leq\al(\te)$ (resp. $u(0,\te)\geq\al(\te)$), for all $\te\in \R^n$.
\item[(ii)] For any $\vp\in \ct_{\si}$ it holds that if $u-\vp$ attains a local maximum (resp. minimum) at $(t_0,\te_0)\in [0,T]\times \R^n$ such that $u_{t_{0}}(\te_0)=\vp_{t_{0}}(\te_0)$, then
\begin{equation*}
    \vp^{t}_{t_0}(\te_0)\leq (\text{resp. } \geq)\;F\lp \vp_{t_{0}}(\te_0),D\vp_{t_0}(\te_0)\rp,
\end{equation*}
\end{itemize}
If $u$ is both a viscosity subsolution and supersolution, we say that $u$ is a viscosity solution of the rough PDE \eqref{eq:spde-integral}.
\end{defn}

\smallskip

Like in \cite{LS98,LSo1}, we favor here the definition of viscosity solutions through the introduction of the set of  test functions $\ct_{\si}$. The point of view of \cite{BBM,BFM,BM} is related instead on the notion of stochastic sub/superjet. This notion stems directly from the expansion of the strong solution to equation \eqref{eq:spde-integral} given by Proposition \ref{class-decom}, and can be stated as follows:

\begin{defn}
Let $z$ be a function from $[0,T]\times\R^n$ into $\R$ and let $(t_0,\te_0)\in [0,T]\times\R^n$. Then the $\si$-subjet (resp. $\si$-superjet) of $z$ at $(t_0,\te_0)$ is the set, denoted by $\mathfrak{P}_\si^{+}z(t_0,\te_0)$ (resp. $\mathfrak{P}_\si^{-}z(t_0,\te_0)$), of
$$
(a_0,p_0,X_0)\in \lp\R\times \R^n\times \R^{n,n}\rp
$$
such that
\begin{align}\label{eq:def-jet}
z_t(\te)\leq (\text{ resp. }\geq\,)&\, z_{t_0}(\te_0)+a_0\, (t-t_0)+ b^l_0\,\der x^l_{t_0t}+c^{lk}_0\,\mathbf{x}^{\mathbf{2},kl}_{t_0t}\notag\\
&+p_0^i\,(\te^i-\te_0^i)+q^{jl}_0\,\der x^l_{t_0t}\,(\te^j-\te_0^j)
+o\lp|t-t_0|+|\te-\te_0|\rp,
\end{align}
where for $1\leq l,k\leq d$ and $1\leq i,j \leq n$, we consider the real valued coefficients:
\begin{align*}
b^l_0&=\si^l(z_{t_0}(\te_0),p_0), \qquad
c^{lk}_0=b_0^k\,\partial_z\si^l(z_{t_0}(\te_0),p_0)
+q_0^{jk}\,\partial_{p^j}\si^l(z_{t_0}(\te_0),p_0),\\
q^{jl}_0&=p_0^j\,\partial_z\si^l(z_{t_0}(\te_0),p_0)
+X_0^{ij}\,\partial_{p^i}\si^l(z_{t_0}(\te_0),p_0).
\end{align*}
\end{defn}

Related to the last definition, one can also define viscosity solutions to rough PDEs by means of the jets:
\begin{defn}\label{def:viscosity-jets}
A function $u\in C([0,T]\times \R^n)$ is called a jet-viscosity subsolution (resp. supersolution) of \eqref{eq:spde-integral} if
\begin{itemize}
\item[(i)] We have $u(0,\te)\leq (\text{ resp. } \geq \;)\;\al(\te)$, for all $\te\in \R^n$;
\item[(ii)] For any $(t_0,\te_0)\in [0,T]\times \R^n$ and any $(a_0,p_0)\in   \mathfrak{P}_\si^{+}u(t_0,\te_0)$ (resp. $\mathfrak{P}_\si^{-}u(t_0,\te_0)$), it holds that
    \begin{equation*}
    a_0\leq (\text{ resp. } \geq\; )\;F\lp u(t_0,\te_0),p_0\rp.
\end{equation*}
\end{itemize}
If $u$ is both a jet-viscosity subsolution and supersolution, we say that $u$ is a jet-viscosity solution of \eqref{eq:spde-integral}.
\end{defn}

\begin{rem}
There is an easy way to relate Definition \ref{def:viscosity-jets} with Definition \ref{def:viscosity-critical-points}. Indeed, each element $(a_0,p_0)\in   \mathfrak{P}_\si^{+}u(t_0,\te_0)$ defines a function $\vp$ in a neighborhood of $(t_0,\te_0)$ by:
\begin{equation*}
\vp_{t}(\te)
= z_{t_0}(\te_0)+a_0\, (t-t_0)+ b^l_0\,\der x^l_{t_0t}+c^{lk}_0\,\mathbf{x}^{\mathbf{2},kl}_{t_0t} 
+p_0^i\,(\te^i-\te_0^i)+q^{jl}_0\,\der x^l_{t_0t}\,(\te^j-\te_0^j).
\end{equation*}
This function can be considered as an element of $\ct_{\si}$, at least locally around $(t_{0},\te_{0})$. It satisfies $\vp_{t_{0}}(\te_{0})=u_{t_{0}}(\te_{0})$, $\vp_{t_{0}}^{t}(\te_{0})=a_{0}$ and $D\vp_{t_{0}}(\te_{0})=p_{0}$. Now, since we start from an element of $\mathfrak{P}_\si^{+}u(t_0,\te_0)$, the function $u-\vp$ admits a local maximum at $(t_0,\te_0)$. If $u$ is assumed to be a subsolution, then according to Definition \ref{def:viscosity-critical-points} we must have $\vp_{t_{0}}^{t}(\te_{0})\le F( \vp_{t_{0}}(\te_0),D\vp_{t_0}(\te_0))$, which is compatible with the condition $a_{0}\le F( \vp_{t_{0}}(\te_0),D\vp_{t_0}(\te_0))$ of Definition \ref{def:viscosity-jets}.

 \smallskip

This heuristic argument is not as easy to formalize as in the deterministic setting, where a complete identification between jets $\mathfrak{P}_\si^{\pm}$ and functions in $\ct_{\si}$ is possible. However, we shall observe this relation on the particular cases of equations below.
\end{rem}

\section{Transport type equation}
\label{sec:transport}
We investigate here the first example of stochastic equation which can be solved thanks to viscosity techniques. This occurs when, going back to equation \eqref{eq:spde-integral}, we take $F(\te,u,Du)$ $=F(\te,Du)$ and $\si^{l}(\te,u,Du)=-\partial_{\te^i}u \, {A}_{l}^{i}(\te)$ for some smooth and bounded vector fields ${A}_{l}$. Specifically, the equation we shall consider here is of the form:
\begin{equation}
\label{eq:transport-intg}
u_t(\te)=\alpha(\te)+\int_0^t F\lp r,\te,Du_r(\te)\rp dr-\int_0^t\partial_{\te^i} u_r(\te)\,{A}_{l}^{i}(\te)\,  dx_r^l,
\end{equation}
with the same notational conventions as in Section \ref{sec:viscosity-solutions}. 
As in \cite{CFO}, equation \eqref{eq:transport-intg} is handled through a composition with the rough flow related to the vector fields ${A}^{l}(\te)\,  dx_r^l$. This is where we will make use of the (presumably) classical results of Section \ref{sec:prelim-stoch-flows}.

\smallskip

The definition of viscosity solution for equation \eqref{eq:transport-intg} is a particular case of Definition \ref{def:viscosity-critical-points}. However, it is important enough to label the explicit expression of the space of test functions we get in this case (which will be called $\ct_{{A}}$ in the remainder of the section) for further use:
\begin{defn}\label{def:test-fct-transport}
Consider a family of smooth and bounded vector fields ${A}_1,\ldots,{A}_d$ defined on $\R^n$.
Let $\psi$ be a strongly controlled process in $\cs_\bx^\ka([0,T];V)$ with $V=\cac^2(\R^n)$ and $\ka>1/3$.
We say that $\psi$ is an element of the test functions $\ct_{{A}}$ if the coefficients in the decomposition~\eqref{eq:strg-ctld-dcp} of $\psi$ satisfy
\begin{equation}\label{eq:dcp-test-transport}
\psi^{x;l} = -{A}_l \psi, \quad \text{and}\quad
\psi^{xx;l_1l_2} = {A}_{l_{1}} {A}_{l_{2}} \psi,
\end{equation}
where we recall that ${A}_{l_1} {A}_{l_2} \psi
={A}_{l_1}^{j_1}\,\partial_{\te^{j_1}}{A}^{j_2}_{l_2}\,\partial_{\te^{j_2}} \psi
+ {A}_{l_1}^{j_1} {A}_{l_2}^{j_2} \, \partial^2_{\te^{j_1}\te^{j_2}} \psi$.
\end{defn}

\begin{rem}
Obviously, this definition is just obtained by particularizing Definition \ref{def:test-fct} to the case $\si^{l}(\te,\psi,p)=-p^{i} \, {A}_l^i(\te)$. Indeed, we have $\partial_{\psi}\si^{l}(\te,\psi,p)=0$ and $\partial_{p_j}\si^{l}(\te,\psi,p)=-{A}_{l}^{j}(\te)$. Plugging this information into \eqref{eq:dcp-test-fct} we end up with
\begin{equation*}
\psi^{xx;l_1l_2}
= \partial_{p^j}\si^{l_1}(\psi,D\psi)\,\partial_{\te^j} [ \si^{l_2}(\psi,D\psi)]
= {A}_{l_{1}}^{j} \,\partial_{\te^j} [ D^{i}\psi \, {A}_{l_{2}}^{i}]
= {A}_{l_{1}} {A}_{l_{2}} \psi.
\end{equation*}
\end{rem}

In order to compare the solution to our equation with the solution of a deterministic equation with random coefficients, we shall also consider the set of test functions $\ct_{0}$ related to the deterministic problem $\partial_t u = F(u,Du)$, which is nothing else than the set of functions with space-time regularity $\cac^{1,2}$.
We now state a proposition which gives a one-to-one correspondence between the set of test functions $\ct_{{A}}$ corresponding to our stochastic problem and the set $\ct_{0}$:
\begin{prop}\label{prop:T(V)-T(0)}
Consider a family of smooth and bounded vector fields ${A}_1,\ldots,{A}_d$ defined on $\R^n$, and the set $\ct_{{A}}$ introduced in Definition \ref{def:test-fct-transport}. Then:

\smallskip

\noindent
\emph{(i)} Consider the solution $\phi$ to equation \eqref{eq:rde-flow}, considered as a flow. Then for any $\psi\in\ct_{{A}}$, the composition $\psi\circ\phi$ is an element of $\ct_{0}$.

\smallskip

\noindent
\emph{(ii)} Consider the solution $\phi^{-1}$ to equation \eqref{eq:rde-invflow}. Then for any $\psi\in\ct_{0}$, the composition $\psi\circ\phi^{-1}$ is an element of $\ct_{{A}}$.
\end{prop}

\begin{proof}
Consider $\psi\in\ct_{{A}}$, which means in particular that $\psi$ is a real valued strongly controlled process, and take the solution $\phi$ to equation \eqref{eq:rde-flow}. According to Proposition \ref{prop:compo-scp}, $\psi\circ\phi$ is still a strongly controlled process, with $[\psi\circ\phi]^{x;l_1}= \lc \psi^{x;l_1} +  {A}_{l_1}^j\,\partial_{\te^j}\psi \rc\circ\phi$ and
\begin{equation*}
[\psi\circ\phi]^{xx;il_1l_2}=
\lc
\psi^{xx;l_1l_2} + {A}_{l_2}^j\, \partial_{\te^j} \psi^{x;l_1} + {A}_{l_1}^j\, \partial_{\te^j}\psi^{x;l_2} + {A}_{l_1}^{j_1} {A}_{l_2}^{j_2}\, \partial^2_{\te^{j_1}\te^{j_2}}\psi + {A}_{l_2}^{k}\,\partial_{\te^k}{A}^j_{l_1}\,\partial_{\te^j}\psi
\rc\circ\phi.
\end{equation*}
Plugging now the expression \eqref{eq:dcp-test-transport} for the components of $\psi$ into this last expression, it is readily checked that both $[\psi\circ\phi]^{x;l_1}$ and $[\psi\circ\phi]^{xx;l_1l_2}$ vanish. This shows item (i).

\smallskip

As regards the proof of item (ii), we apply again Proposition \ref{prop:compo-scp}, plus the fact that $\psi^{x;l}=0$ and $\psi^{xx;l_1l_2}=0$.
\end{proof}

\smallskip

\begin{rem}
As a simple corollary of Proposition \ref{prop:T(V)-T(0)}, we can answer a question raised by Remark \ref{rmk:nonempty-T(sigma)}. Indeed, we can now assert that $\ct_{{A}}$ is as rich as $\ct_{0}$, which contains all $\cac^{1}$ functions in time.
\end{rem}

\smallskip

We can now relate our rough equation \eqref{eq:transport-intg} to a deterministic problem in the following way:
\begin{prop}\label{prop:transport-noisy-determ}
A function $u\in \cac([0,T]\times \R^n)$ is  a viscosity subsolution (resp. supersolution) of \eqref{eq:transport-intg} if and only if, setting $\hu=u\circ\phi^{{A}}$, the function $\hu$ is  a viscosity subsolution (resp. supersolution) of the following equation:
\begin{equation}\label{eq:deterministic-transport}
\hu_t(\te)=\alpha(\te)+\int_0^t \hf\lp r,\te,D\hu_r(\te)\rp dr,
\end{equation}
where we have used the notation
\begin{equation*}
\hf\lp r,\te, p\rp = F \lp r, \phi_r(\te), \lla p , \, [D\phi_r^{-1}]_{\mid\phi_r(\te)} \rra \rp,
\quad\text{with}\quad
\lla p , \, [D\phi_r^{-1}]_{\mid\phi_r(\te)} \rra \equiv
p^{j} \, \lc \partial_{\te^j} \phi_r^{-1}\rc \!\!(\phi_r(\te)).
\end{equation*}
\end{prop}

\begin{proof}
Assume first that the function $u\in \cac([0,T]\times \R^n)$ is  a viscosity subsolution of \eqref{eq:spde-integral}. This means that for any $\psi\in\ct_{{A}}$ condition (D) is met, with:
\begin{itemize}
\item[\textbf{(D)}] Whenever $u-\psi$ reaches a maximum at $(t_0,\te_0)$ with the additional assumption $u_{t_0}(\te_0)=\psi_{t_0}(\te_0)$, then $\psi_{t_0}^{t}(\te_0)\le F(t_0,\te_0,D\psi_{t_0}(\te_0))$.
\end{itemize}
Consider now $\phi=\phi^{{A}}$ defined by \eqref{eq:rde-flow} and $\hu=u\circ\phi$. We wish to show that $\hu$ is a viscosity solution of equation \eqref{eq:deterministic-transport}. This problem can be reduced to the following one: show that for any $\vp\in\ct_{0}$ such that $\hu-\vp$ reaches a maximum at $(t_0,\te_0)$ with the additional assumption $\hu_{t_0}(\te_0)=\vp_{t_0}(\te_0)$, then $\vp_{t_0}^{t}(\te_0)\le \hf(t_0,\te_0,D\vp_{t_0}(\te_0))$.

\smallskip

In order to prove this last claim, consider $\psi\in\ct_{{A}}$ and set $\hpsi=\psi\circ\phi^{{A}}$. According to Proposition \ref{prop:T(V)-T(0)}, we have $\hpsi\in\ct_{0}$. Moreover, condition (D) can be translated as: if $\hu\circ\phi^{-1}-\hpsi\circ\phi^{-1}$ reaches a maximum at $(t_0,\te_0)$ with the additional assumption $[\hu\circ\phi^{-1}]_{t_0}(\te_0)=[\hpsi\circ\phi^{-1}]_{t_0}(\te_0)$, then
\begin{equation*}
[\hpsi\circ\phi^{-1}]_{t_0}^{t}(\te_0)\le F(t_0,\te_0,D[\hpsi\circ\phi^{-1}]_{t_0}(\te_0))
= \hf(t_0,\te_0,[D\hpsi_{t_0}]\circ\phi^{-1}(\te_0)).
\end{equation*}
Owing to the fact that  $\phi_{t}$ is a diffeomorphism for any $t\in\R_{+}$, we can recast the latter condition as: if $\hu-\hpsi$ reaches a maximum at $(t_0,\te_0)$ with the additional assumption $\hu_{t_0}(\te_0)=\hpsi_{t_0}(\te_0)$, then $\hpsi_{t_0}^{t}(\te_0)\le \hf(t_0,\te_0,D\hpsi_{t_0}(\te_0))$, this assertion being true for any element $\hpsi$ of the form $\psi\circ\phi$ with $\psi\in\ct_{{A}}$. Thanks to Proposition \ref{prop:T(V)-T(0)}, the collection of these elements coincides with $\ct_{0}$, which proves that $\hu=u\circ\phi$ is a viscosity subsolution to equation \eqref{eq:deterministic-transport}.

\smallskip

The other relations are proven exactly along the same lines, and are left to the reader as an easy exercise.

\end{proof}

\begin{rem}
With the same kind of considerations, one can also prove that strong solutions to equation \eqref{eq:transport-intg} are also viscosity solutions.
\end{rem}

We can now turn to the main aim of this section, namely an existence and uniqueness result for the solution to equation \eqref{eq:transport-intg}:
\begin{thm}
Consider a family of smooth and bounded vector fields ${A}_1,\ldots,{A}_d$ defined on $\R^n$, and a bounded Lipschitz function $F:[0,T]\times\R^n\times\R^n$. Then equation \eqref{eq:transport-intg} admits a unique viscosity solution in the sense of Definition~\ref{def:viscosity-critical-points}.
\end{thm}

\begin{proof}
Thanks to Proposition \ref{prop:transport-noisy-determ}, the existence and uniqueness problem for equation~
\eqref{eq:transport-intg} is equivalent to the existence and uniqueness problem for equation~\eqref{eq:deterministic-transport}. In order to solve the latter problem one can invoke the classical viscosity theory, which amounts to show \emph{comparison} properties for $\hf$ as shown in \cite{CFO}. This condition is ensured whenever $F$ is bounded and Lipschitz, which ends the proof.

\end{proof}

\begin{rem}
A wide range of examples for the function $F$ are investigated in \cite{CFO} (see also~\cite{DFG}), and would lead to the same results in our setting. We haven't delved deeper into this direction since our aim is to settle a clear formulation for the notion of viscosity solution to rough PDEs rather than getting new results in terms of existence and uniqueness.
\end{rem}

We close this section by an expression of the solution to \eqref{eq:transport-intg} starting from jets. We will first particularize the definition of jets to our linear context for sake of clarity.
\begin{defn}
Let $z$ be a function from $[0,T]\times\R^n$ into $\R$ and let $(t_0,\te_0)\in [0,T]\times\R^n$. Then the subjet of $z$ at $(t_0,\te_0)$ related to equation \eqref{eq:transport-intg} is the set, denoted by $\mathfrak{P}_{A}^{+}z(t_0,\te_0)$, of triples
$$
(a_0,p_0,X_0)\in \lp\R\times \R^n\times \R^{n,n}\rp
$$
such that \eqref{eq:def-jet} holds, where for $1\leq l,k\leq d$ and $1\leq i,j \leq n$, we consider the real valued coefficients:
\begin{align*}
&b^l_0 =-p_{0}^{i} \, A_{l}^{i}(\te_{0}), \qquad
c^{lk}_0=\lc X_0^{ij}\,  A_{k}^{i}(\te_{0})+p_0^i\,\partial_{\te^j}A_{k}^i(\te_0)\rc  A_{l}^{j}(\te_{0}), \\
&q^{jl}_0
=
-\lc X_0^{ij}\,  A_{l}^{i}(\te_{0})+p_0^i\,\partial_{\te^j}A_{l}^i(\te_0)\rc.
\end{align*}
The corresponding superjet is defined accordingly.
\end{defn}

Consider now the jets $\mathfrak{P}_{0}^{\pm}z(t_0,\te_0)$ related to the deterministic equation \eqref{eq:deterministic-transport}, corresponding to $A=0$. We are able to relate stochastic and deterministic jets in the following way:

\begin{prop}
Let $u^{A}$ be the unique viscosity solution to equation \eqref{eq:transport-intg} and let $u^{0}$ be its deterministic counterpart, solution to \eqref{eq:deterministic-transport}. Then or all couples $(t_0,\te_0)\in[0,T]\times\R^{n}$ there exists a bijective correspondence between $\mathfrak{P}_{A}^{\pm}u^{A}(t_0,\te_0)$  and $\mathfrak{P}_0^{\pm}u^{0}(t_0,\te_0)$.
\end{prop}

\begin{proof}
Let us regularize the coefficients and the noise of equation \eqref{eq:transport-intg}, producing a strong solution $u^{\ep,A}$. The same can be done for the deterministic equation, and we get a strong solution $u^{\ep,0}$. Let us also consider the solution $\phi$ to equation \eqref{eq:rde-flow}.
Then applying Corollary \ref{cor:compo-scp} we obtain a Taylor-type expansion for the composition $u^{\ep}\circ \phi^{\ep}$, whose coefficients satisfy:
\begin{equation*}
[u^{\ep}\circ \phi^{\ep}]_{t_0}^t(\te)=u^{\ep,t}(\phi^{\ep}_{t_0}(\te_0)), \qquad
D[u^{\ep}\circ \phi^{\ep}]_{t_0}(\te_0):=\partial_{\te^i}u^{\ep}_{t_0}(\phi^{\ep}_{t_0}(\te_0))
D\phi^{\ep,i}_{t_0}(\te_0)
\end{equation*}
and
\begin{equation*}
[u^{\ep}\circ \phi^{\ep}]_{t_0}^{x;l_1}(\te_0)=[u^{\ep}\circ \phi^{\ep}]_{t_0}^{xx;l_1l_2}(\te_0)
=D[u^{\ep}\circ \phi^{\ep}]_{t_0}^{x;l_1}(\te_0)=0.
\end{equation*}
The composition with $\phi$ transforms thus the unique element $(a_0,p_0,X_0)$ of $\mathfrak{P}_{A}^{\pm}u^{\ep,A}(t_0,\te_0)$ into the unique element of $\mathfrak{P}_{0}^{\pm}u^{\ep,0}(\phi(t_0,\te_0))$.

\smallskip

Reciprocally, let us strart from the solution $u^{\ep,0}$ of the regularized deterministic equation. Recall that  $\zeta$ designates the solution to equation \eqref{eq:rde-invflow}, and let us call $\zeta^{\ep}$ its regularization. Then, invoking  again Corollary \ref{cor:compo-scp} we obtain a Taylor-type expansion for the composition $u^{\ep,0}\circ \zeta^{\ep}$ with the corresponding coefficients:
\begin{align*}
&[u^{\ep,0}\circ \zeta^{\ep}]_{t_0}^t(\te_0)=u^{\ep,0,t}(\zeta^{\ep}_{t_0}(\te_0)),
\qquad
[u^{\ep,0}\circ \zeta^{\ep}]_{t_0}^{x;l_1}(\te_0)=-A_{l_1}[u^{\ep,0}\circ \zeta^{\ep}]_{t_0}(\te)\\
&[u^{\ep,0}\circ \zeta^{\ep}]_{t_0}^{xx;l_1l_2}(\te_0)=A_{l_1}A_{l_2}[u^{\ep,0}\circ \zeta^{\ep}]_{t_0}(\te_0),\quad
D[u^{\ep,0}\circ \zeta^{\ep}]_{t_0}(\te_0):=\partial_{\te^i}u^{\ep,0}_{t_0}(\zeta^{\ep}_{t_0}(\te_0))D\zeta^{\ep,i}_{t_0}(\te_0),
\end{align*}
and
\begin{equation*}
D[u^{\ep,0}\circ \zeta^{\ep}]_{t_0}^{x;l_1}(\te_0)
=\partial^2_{\te^i\te^j}u^{\ep,0}_{t_0}(\zeta^{\ep}_{t_0}(\te_0))D\zeta^{\ep;j}_{t_0}(\te_0)\zeta_{t_0}^{\ep,x;il_1}(\te_0)+\partial_{\te^i}u^{\ep,0}_{t_0}(\zeta^{\ep}_{t_0}(\te_0))D\zeta_{t_0}^{\ep,x;il_1}(\te_0).
\end{equation*}
Therefore, expanding the terms $[u^{\ep,0}\circ \zeta^{\ep}]_{t_0}^{x;l_1}(\te_0)$ and $[u^{\ep,0}\circ \zeta^{\ep}]_{t_0}^{xx;l_1l_2}(\te_0)$ in terms of the derivatives with respect to $\te$ of $\zeta^{\ep}$ up to order 2, and in terms of the vector fields $A_l$ and their first derivatives, we deduce that if $(a_0,p_0,X_0)\in\mathfrak{P}_{0}^{\pm}u^{\ep,0}(t_0,\te_0)$  then $(a_0,p_0^i\,D\zeta^{\ep;i}_{t_0}(\te_0),\tilde{X}_0)\in\mathfrak{P}_{A}^{\pm}u^{\ep,A}(t_0,\te_0)$, where
\begin{equation*}
\tilde{X}_0^{j_1j_2}:=X_0^{j_1'j_2'}\partial_{\te^{j_2}}\zeta_{t_0}^{\ep;j_2'}(\te_0)\partial_{\te^{j_1}}\zeta_{t_0}^{\ep,j_1'}(\te_0)+
p_0^{j_1'}\partial^2_{\te^{j_1}\te^{j_2}}\zeta_{t_0}^{\ep,j_1'}(\te_0).
\end{equation*}
We have thus identified the jets $\mathfrak{P}_{0}^{\pm}u^{\ep,0}(t_0,\te_0)$ and $\mathfrak{P}_{A}^{\pm}u^{\ep,A}(t_0,\te_0)$ of the regularized equations. Taking limits on $\mathfrak{P}_{0}^{\pm}u^{\ep,0}(t_0,\te_0)$, this allows to identify the jets related to $u^{0}$ and $u^{A}$.

\end{proof}

\begin{rem}
Obviously, the fact that we resort to solutions of equations in order to establish the jet correspondence is less attracting than the result concerning test functions. We shall thus privilege the definition of viscosity solution through the sets $\ct_{\si}$ of Definition \ref{def:test-fct} in  the sequel.
\end{rem}

\section{Equation with semilinear stochastic dependence}
\label{sec:semilinear}

We investigate here a second example of stochastic equation which can be solved thanks to viscosity techniques. This occurs when, going back to equation \eqref{eq:spde-integral}, we take $F(\theta,u,Du)=F\lp Du \rp$ and $\sigma^l(\theta,u,Du)=H_l(u)$ for some smooth, bounded and nonlinear functions $H_l:\R\to\R$.  Specifically, the equation we shall consider here is of the form:
\begin{equation}
\label{eq:spde-integral-semilinear}
u_t(\te)=\alpha(\te)+\int_0^t F\lp Du_r(\te)\rp dr+\sum_{l=1}^{d}\int_0^t H_l (u_r(\te))\,  dx_r^l,
\end{equation}
with the same notational conventions as in Sections \ref{sec:viscosity-solutions} and \ref{sec:transport}. We shall also follow the same notational conventions and the strategy considered by Lions and Souganidis in \cite{LSo1}. In particular, we shall rely on the following equation, which is just equation \eqref{eq:rde-flow} where $A_l$ has been replaced by $H_l$:
\begin{equation}\label{eq:rde-flow-H}
\phi_{t}(v) = v + \sum_{l=1}^{m} \int_{0}^{t} H_{l}(\phi_{u}(v)) \, dx_{u}^{l} .
\end{equation}
This equation will be interpreted as a flow from $\R$ to $\R$, and its inverse $\zeta$ is obviously defined by equation \eqref{eq:rde-invflow} where $A_l$ has been replaced by $H_l$.

\smallskip

As we did in Section \ref{sec:transport}, we give the explicit expression of the space of test functions we get in this case (which will be called $\ct_{{H}}$ in the remainder of the section) for further use:

\begin{defn}\label{def:test-fct-semilinear}
Consider a family of smooth, bounded and nonlinear vector fields ${H}_1,\ldots,{H}_d$ defined on $\R$.
Let $\psi$ be a strongly controlled process in $\cs_\bx^\ka([0,T];V)$ with $V=\cac^2(\R^n)$ and $\ka>1/3$.
We say that $\psi$ is an element of the test functions $\ct_{H}$ if the coefficients in the decomposition \eqref{eq:strg-ctld-dcp} of $\psi$ satisfy
\begin{equation}\label{eq:dcp-test-semilinear}
\psi^{x;l} = H_l\circ \psi, \quad \text{and}\quad
\psi^{xx;l_1l_2} = \lc H_{l_2}H_{l_1}\rc \circ \psi.
\end{equation}
\end{defn}

\smallskip

In order to compare the solution to our equation with the solution of a deterministic equation with random coefficients, we shall also consider the set of test functions $\ct_{0}$ related to the deterministic problem $\partial_t v =\tilde{F}(t,v,Dv)$ (where $\tilde{F}$ is a suitable transformation of  $F$ to be specified later on), which is nothing else than the set of functions with space-time regularity $\cac^{1,2}$.
As in Proposition \ref{prop:T(V)-T(0)}, we establish now a one-to-one correspondence between the set of test functions $\ct_{H}$ corresponding to our stochastic problem and $\ct_{0}$:
\begin{prop}\label{prop:semilinear:T(V)-T(0)}
Consider a family of smooth, bounded and nonlinear vector fields $H_1,\ldots,$ $H_d$ defined on $\R$, and the set $\ct_{H}$ introduced in Definition \ref{def:test-fct-semilinear}. Then:

\smallskip

\noindent
\emph{(i)} Let $\phi$ be the solution to equation \eqref{eq:rde-flow-H}, considered as a flow. Then for any $\psi\in\ct_{0}$, the composition $\phi\circ\psi$ is an element of $\ct_{H}$.

\smallskip

\noindent
\emph{(ii)} Let $\zeta$ be the inverse of $\phi$, solution to equation \eqref{eq:rde-invflow} with $A_l=H_l$. Then for any $\psi\in\ct_{H}$, the composition $\zeta\circ\psi$ is an element of $\ct_{0}$.
\end{prop}

\begin{proof}
 Take $\psi\in\ct_{0}$, which means in particular that $\psi$ is a real valued strongly controlled process with $\psi^x=0$ and $\psi^{xx}=0$. Consider the flow $\phi$ defined by equation \eqref{eq:rde-flow-H}. According to Proposition \ref{prop:compo-scp}, $\phi\circ\psi$ is still a strongly controlled process, with $[\phi\circ\psi]^{x;l_1}(\te)= H_{l_1}([\phi\circ\psi](\te))$, $[\phi\circ\psi]^{xx;il_1l_2}=
[H_{l_2}H_{l_1}]([\phi\circ\psi](\te))$ and $[\phi\circ\psi]^{t}(\te)=\partial_{v}\phi(\psi(\te))\,\psi^{t}(\te)$. This shows item (i).

\smallskip

Concerning the proof of item (ii), take $\psi\in\ct_{H}$. We also recall that for the flow $\zeta$ related to to equation \eqref{eq:rde-invflow} with $A_l=H_l$, we have:
\begin{align*}
&\zeta^{x;l_1}(v)=-H_{l_1}(v)\,\partial_v \zeta(v),\quad \zeta^{t}(v)=0,\\
&\zeta^{xx;l_1l_2}(v)=H_{l_1}(v)\lc\partial_v H_{l_2}(v)\,\partial_v\zeta(v)+H_{l_2}(v)\,\partial_v^2\zeta(v)\rc,
\end{align*}
where here $v\in\R$. Then, applying again Proposition \ref{prop:compo-scp}, one easily gets that $\zeta\circ\psi$ is a strongly controlled process with
\begin{equation*}
[\zeta\circ\psi]^{x;l_1}(\te)=[\zeta\circ\phi]^{xx;l_1l_2}(\te)= 0,
\quad\mbox{and}\quad
[\zeta\circ\psi]^{t}(\te)=\partial_{v}\zeta(\psi(\te))\,\psi^{t}(\te).
\end{equation*}
This means that $\zeta\circ\psi\ct_{0}$ and finishes the proof of our one-to-one correspondence.
\end{proof}

\smallskip

In order to relate our rough equation \eqref{eq:spde-integral-semilinear} to a deterministic problem, let us label the following easy monotonicity lemma:

\begin{lem}\label{lem:increasing-phi}
Let $\phi$ be the flow related to equation~\eqref{eq:rde-flow-H}. Then $\phi:\R\to\R$ is a strictly  increasing function.
\end{lem}

\begin{proof}

The process $\phi$, seen as a function from $[0,T]\times\R$ to $\R$, is differentiable in $v$ (we refer to \cite{FV-bk} for a precise account on this fact). If we set $J_{t}(v)=\partial_{v}\phi_{t}(v)$, $J$ is solution to the following equation:
\begin{equation*}
J_{t}(v) = 1 + \sum_{l=1}^{m} \int_{0}^{t} H_{l}^{\prime}(\phi_{u}(v)) \, J_{u}(v) \, dx_{u}^{l} .
\end{equation*}
This equation admits an explicit solution, namely:
\begin{equation*}
J_{t}(v) =
\exp\lp  \sum_{l=1}^{m} \int_{0}^{t} H_{l}^{\prime}(\phi_{u}(v))  \, dx_{u}^{l}\rp,
\end{equation*}
which is obviously strictly positive. Thus $\phi$ is strictly increasing as a function of $v$.
\end{proof}

\smallskip

The relation between our rough equation \eqref{eq:spde-integral-semilinear} and a deterministic problem now takes the following form:
\begin{prop}\label{prop:semilinear-noisy-determ}
Let $\phi$ the $\mathcal{C}^3$-diffeomorphism defined by \eqref{eq:rde-flow-H} with $A_l=H_l$. Set $u=\phi \circ \tu$ where $\tu$ satisfies the equation:
\begin{equation}\label{eq:deterministic-semilinear}
\tu_t(\te)=\alpha(\te)+\int_0^t \tf\lp r,\tu_r(\te),D\tu_r(\te)\rp dr,
\end{equation}
where we have used the notation
\begin{equation*}
\tf\lp r,v,p\rp:=\frac{1}{\partial_v \phi_r(v)}\,F\lp\partial_v \phi_r(v)p\rp.
\end{equation*}
Then a function $u\in \cac([0,T]\times \R^n)$ is  a viscosity subsolution (resp. supersolution) of \eqref{eq:spde-integral-semilinear} if and only if, $\tu$ is  a viscosity subsolution (resp. supersolution) of \eqref{eq:deterministic-semilinear}.

\end{prop}

\begin{proof}

Assume first that the function $u\in \cac([0,T]\times \R^n)$ is  a viscosity subsolution of \eqref{eq:spde-integral-semilinear}. We want to show that $\tu$ is a viscosity solution of equation \eqref{eq:deterministic-semilinear}. Thus, we need to show that for any $\vp\in\ct_{0}$ such that $\tu-\vp$ reaches a local maximum at $(t_0,\te_0)$ with the additional assumption $\tu_{t_0}(\te_0)=\vp_{t_0}(\te_0)$, then $\vp_{t_0}^{t}(\te_0)\le \tf(t_0,\te_0,D\vp_{t_0}(\te_0))$.

\smallskip

In order to prove this, fix $\vp\in\ct_{0}$ such that $\tu-\vp$ reaches a local maximum at $(t_0,\te_0)$ with the additional assumption $\tu_{t_0}(\te_0)=\vp_{t_0}(\te_0)$. Then, notice that we have $\tu_{t}(\te)\geq \vp_t(\te)$ in some neighborhood of $(t_0,\te_0)$. Owing to Lemma \ref{lem:increasing-phi}, we have $u_{t}(\te)\geq [\phi\circ\vp]_t(\te)$ in some neighborhood of $(t_0,\te_0)$. By Proposition \ref{prop:semilinear:T(V)-T(0)} we have that $\phi\circ\vp$ is an element of $\ct_{H}$. Therefore, we deduce that
$u-\phi\circ\vp$ reaches a local maximum at $(t_0,\te_0)$ with the additional assumption $u_{t_0}(\te_0)=[\phi\circ\vp]_{t_0}(\te_0)$. Since $u$ is a viscosity subsolution then
\begin{equation}
\label{sub:cond:sto:eq}
[\phi\circ\vp]_{t_0}^{t}(\te_0)\le F(D[\phi\circ\vp]_{t_0}(\te_0)).
\end{equation}
Notice that we can write $[\phi\circ\vp]_{t_0}^{t}(\te_0)=\partial_v\phi_{t_0}(\vp_{t_0}(\te_0))\,\vp^t_{t_0}(\te_0)$, with $\partial_v\phi_{t_0}(\vp_{t_0}(\te_0))>0$, and $F(D[\phi\circ\vp]_{t_0}(\te_0))=F(\partial_v\phi_{t_0}(\vp_{t_0}(\te_0))D\vp_{t_0}(\te_0))$. Plugging these expressions into \eqref{sub:cond:sto:eq} we obtain the equivalent inequality
\begin{equation*}
\vp^t_{t_0}(\te_0)\le \frac{1}{\partial_v\phi_{t_0}(\vp_{t_0}(\te_0))}F(\partial_v\phi_{t_0}(\vp_{t_0}(\te_0))D\vp_{t_0}(\te_0))=:\tf(t_0,\te_0,D\vp_{t_0}(\te_0)),
\end{equation*}
showing the desired property for $\tu$.

\smallskip

The other relations are proven exactly along the same lines, and are left to the reader for sake of conciseness.
\end{proof}

We can now turn to the main aim of this section, namely an existence and uniqueness result for the solution to equation \eqref{eq:spde-integral-semilinear}:
\begin{thm}
Let $\alpha\in {\rm BUC}(\R^n)$. Consider a family of smooth, bounded and nonlinear vector fields $H_1,\ldots,H_d$ defined on $\R^n$, and a bounded Lipschitz function $F$ defined on $\R^n$ verifying that:
\begin{itemize}
\item[\textbf{(M)}] There exists a positive constant $C$ such that, either
\begin{itemize}
\item[(a)] $p\cdot DF(p)-F\leq C$ or

\smallskip

\item[(b)] $p\cdot DF(p)-F\geq -C$,
\end{itemize}
\end{itemize}
for a.e. $p\in \R^n$. Then equation \eqref{eq:spde-integral-semilinear}  admits a unique viscosity solution in the sense of Definition~\ref{def:viscosity-critical-points}.
\end{thm}

\begin{proof}
Thanks to Proposition \ref{prop:semilinear-noisy-determ}, the existence and uniqueness problem for equation~\eqref{eq:spde-integral-semilinear} is equivalent to the existence and uniqueness problem for equation~\eqref{eq:deterministic-semilinear}. In order to solve the latter problem one can invoke the classical viscosity theory using the results sketched in \cite{LSo1}. This is the reason why we need to assume $F$  bounded, Lipschitz and satisfying \textbf{(M)}.
\end{proof}


\begin{thebibliography}{99}


\bibitem{BBM}
Buckdahn, R., Bulla, I., Ma, J.:
\emph{Pathwise Taylor expansions for Itô random fields.}
Math. Control Relat. Fields {\bf 1} (2011), no. 4, 437--468.

\bibitem{BFM}
Buckdahn, R.,  Ma, J., Zhang, J.:
\emph{Pathwise Taylor Expansions for Random Fields on Multiple Dimensional Paths.}
arXiv:1310.0517v1 [math.PR] 1 Oct 2013.



\bibitem{BM}
Buckdahn, R., Ma, J.: \emph{Pathwise stochastic Taylor expansions and stochastic viscosity solutions for fully nonlinear stochastic PDEs.} Ann. Probab., Vol. \textbf{30}, No. 3, (2002), 1131-1171.

\bibitem{CFO}
Caruana, M., Friz, P.,  Oberhauser, H.: \emph{A (rough) pathwise approach to fully non-linear
stochastic partial differential equations.} Annals IHP (C), Nonlinear Analysis, \textbf{28} (2011), 27-46.

\bibitem{CHLT}
Cass, T., Hairer, M., Litterer, M., Tindel, S.:
\emph{H\"{o}rmander's theorem for Gaussian rough differential
equations.}
To appear in Ann. Probab.

\bibitem{CDFO}
 Crisan, D.,  Diehl, J., Friz, P., Oberhauser, H. : \emph{Robust Filtering: Correlated Noise and Multidimensional Observation}
Annals of Applied Probability \textbf{23} (2013), No. 5, 2139--2160.

\bibitem{DF}
 Diehl, J., Friz, P.: \emph{Backward stochastic differential equations with rough drivers.}	
 Annals of Probability \textbf{40} (2012), 1715--1758.

\bibitem{DFG}
Diehl, J., Friz, P., Gassiat, P.:
\emph{Stochastic control with rough paths.}
arXiv:1303.7160.


\bibitem{FO}
 Friz, P.,  Oberhauser, H.: \emph{Rough path stability of (semi-)linear SPDEs.}	
  Probab. Theory Relat. Fields (2013).


 \bibitem{FO2}
  Friz, P.,  Oberhauser, H.: \emph{On the splitting-up method for rough (partial) differential equations}
Journal of Differential Equations, \textbf{251} (2011), Issue 2,  316--338.


\bibitem{FV-bk}
Friz, P., Victoir, N.:
\emph{Multidimensional dimensional processes seen as rough paths.}
Cambridge University Press (2010).

\bibitem{Gu}
Gubinelli, M.:
\emph{Controlling rough paths.}
J. Funct. Anal. \textbf{216}, (2004), 86-140.

\bibitem {Gu10}
Gubinelli, M.:
\emph{Ramification of rough paths.}
J. Differential Equations 248 (2010), 693-721.

\bibitem{GT}
Gubinelli, M., Tindel, S.: \emph{Rough evolution equations.} Ann. Probab., Vol. \textbf{38}, No. 1, (2010), 1-75.

\bibitem{HP11}
Hairer, M., Pillai, N. S.: \emph{Regularity of Laws and Ergodicity of Hypoelliptic SDEs
Driven by Rough Paths.} arXiv:1104.5218v1 [math.PR] 27 Apr 2011.

\bibitem{Ku-Flour}
Kunita, H.:
Stochastic differential equations and stochastic flows of diffeomorphisms.
\`Ecole d'\'et\'e de probabilit\'es de Saint-Flour, XII 1982, 143--303,
{\it Lecture Notes in Math.} {\bf 1097}, Springer, Berlin, 1984.


\bibitem{LS98}
 Lions, P.-L., Souganidis, P.:
 \emph{Fully nonlinear stochastic partial differential equations. }
 C. R. Acad. Sci. Paris Ser. I Math. \textbf{326} (1998), no. 9, 1085-1092.

\bibitem{LSo1}
Lions, P.-L., Souganidis, P.:
\emph{Fully nonlinear stochastic pde with semilinear stochastic dependence.}
C. R. Acad. Sci. Paris Ser. I Math. \textbf{331} (2000), no. 10, 617-624.

\bibitem{LSo2}
Lions, P.-L., Souganidis, P.:
\emph{Uniqueness of weak solutions of fully nonlinear stochastic partial differential equations.}
C. R. Acad. Sci. Paris Sér. I Math. \textbf{331} (2000), no. 10, 783-790.

\bibitem{N}
Norris, J.: \emph{Simplified Malliavin calculus.} In: Séminaire de Probabilités, XX, 1984/85,
vol. \textbf{1204}, Lecture Notes in Math., 101-130. Springer, Berlin, 1986.

\bibitem{TT}
Tindel, S., Torrecilla, I.:
\emph{Some differential systems driven by a fBm with Hurst parameter greater
than $1/4$.}
In: Stochastic Analysis and
Related Topics. In Honour of Ali Süleyman Üstünel, Paris, June 2010.
Springer Proceedings in Mathematics \& Statistics. Volume 22 (2012).

\end{thebibliography}
\end{document}